\documentclass[11pt,reqno]{amsart}
\usepackage{amsfonts,amssymb,amsmath,amsthm}
\usepackage[colorlinks,citecolor=blue,urlcolor=black,linkcolor=black]{hyperref}

\newtheorem{thm}{Theorem}[section]
\newtheorem{mainthm}{Theorem}

\newtheorem{lemma}[thm]{Lemma}
\newtheorem{claim}{Claim}[thm]
\newtheorem{fact}[thm]{Fact}
\newtheorem{corollary}[thm]{Corollary}
\newtheorem{proposition}[thm]{Proposition}

\theoremstyle{definition}
\newtheorem{defn}[thm]{Definition}
\newtheorem{q}{Question}

\theoremstyle{remark}
\newtheorem{remark}[thm]{Remark}

\renewcommand{\mid}{\mathrel{|}\allowbreak}
\renewcommand{\restriction}{\mathbin\upharpoonright}

\DeclareMathOperator{\ord}{On}
\DeclareMathOperator{\acc}{acc}
\DeclareMathOperator{\id}{id}

\DeclareMathOperator{\cf}{cf}
\DeclareMathOperator{\club}{Club}

\DeclareMathOperator{\dom}{dom}

\DeclareMathOperator{\h}{ht}
\DeclareMathOperator{\ult}{Ult}

\DeclareMathOperator{\ind}{ind}
\DeclareMathOperator{\add}{\mathrm{Add}}
\DeclareMathOperator{\I}{\mathbb I}

\newcommand\s{\subseteq}
\newcommand*\axiomfont[1]{\textsf{\textup{#1}}}
\newcommand\zfc{\axiomfont{ZFC}}
\newcommand\gch{\axiomfont{GCH}}

\begin{document}
\title[Ketonen's question]{Ketonen's question and other cardinal sins}

\author{Assaf Rinot}
\address{Department of Mathematics, Bar-Ilan University, Ramat-Gan 5290002, Israel.}
\urladdr{http://www.assafrinot.com}

\author[Zhixing You]{Zhixing You}
\address{Department of Mathematics, Bar-Ilan University, Ramat-Gan 5290002, Israel}
\email{zhixingy121@gmail.com}

\author[Jiachen Yuan]{Jiachen Yuan}
\address{School of Mathematics, University of Leeds, Leeds LS2 9JT, UK}
\email{yuanjachen@gmail.com}

\date{Preprint as of November 27, 2025. For updates, visit \textsf{http://p.assafrinot.com/69}.}

\begin{abstract}
Answering a question of Ketonen from the late 1970's,
it is proved that a weakly compact cardinal carrying an indecomposable ultrafilter need not be measurable.

The result is obtained by analyzing the limit of a decreasing sequence of models of $\zfc$.
The utility of this proof technique is demonstrated further in this paper,
where a problem by Bagaria and Magidor concerning strong compactness,
and a problem by Lambie-Hanson and Rinot concerning the $C$-sequence number are solved as well.
\end{abstract}
\subjclass[2010]{03E05, 03E35, 03E55}
\keywords{Trees with ascent path, commutative projection system, $C$-sequence number, $\lambda$-strong compactness, intersection model}
\maketitle

\section{Introduction}
\subsection{Compactness principles}
A fundamental theme in set theory is the investigation of compactness
principles of uncountable cardinals.
This includes classical large cardinal notions such as the following four (where arrow denotes logical implication):
$$\text{supercompact}\rightarrow\text{strongly compact}\rightarrow\text{measurable}\rightarrow\text{weakly compact},$$
as well as finer compactness principles which we will be discussing momentarily.
A milestone result in this vein is Magidor's \emph{identity crisis} theorem \cite{MR429566}
asserting that the least strongly compact cardinal may be as small as the first measurable cardinal
and as large as the first supercompact cardinal.

Our first example of a refined compactness principle is a weakening of measurability.
Recall that an uncountable cardinal $\kappa$ is \emph{measurable} iff there exists a uniform $\kappa$-complete ultrafilter $U$ over $\kappa$.
Note that in this case, for every function $f:\kappa\rightarrow\mu$ with $\mu$ a cardinal smaller than $\kappa$,
there exists a set $X\in U$ on which $f$ is constant.
Following Keisler (see \cite{MR0214474}) and Prikry \cite{prikrythesis}, an ultrafilter $U$ over $\kappa$ is \emph{indecomposable}
iff for every function $f:\kappa\rightarrow\mu$ with $\mu<\kappa$,
there exists a set $X\in U$ on which $f$ takes only countably many values.
Silver \cite{silver} asked whether a strongly inaccessible $\kappa$ carrying a uniform indecomposable ultrafilter is necessarily measurable,
and Ketonen \cite{MR585552} gave a partial answer in the affirmative direction, showing that if $\kappa$ is moreover a weakly compact cardinal,
then it must be a Ramsey cardinal.
Motivated by this finding, he asked:
\begin{q}[Ketonen, \cite{MR585552}]\label{qketonen}
Suppose that $\kappa$ is a weakly compact cardinal carrying a uniform indecomposable ultrafilter. Must $\kappa$ be measurable?
\end{q}

A negative answer to Silver's question was soon given by Sheard \cite{Sheard83}, but Ketonen's question remained open.
In this paper we answer it in the negative.

\begin{mainthm}\label{thma} Assuming the consistency of a measurable cardinal, it is consistent that a weakly compact cardinal carries an indecomposable ultrafilter, yet it is not measurable.
\end{mainthm}

Our second example is a weakening of strong compactness,
namely, the notion of \emph{$\lambda$-strong compactness} due to Bagaria and Magidor (see Definition~\ref{dl} below).
As demonstrated in \cite{BM2014,BMO2014,Usu20},
this notion successfully captures a gallery of natural compactness properties occurring in various areas of mathematics,
including Algebra and Graph theory.
In their study,
Bagaria and Magidor \cite[Theorem 6.1]{BM2014} proved that the least $\lambda$-strongly compact cardinal may be a singular cardinal,
so, in particular, it need not be a large cardinal in the classical sense.
Then Gitik \cite{Git20} proved that the least $\lambda$-strongly compact cardinal need not be a large cardinal in the classical sense while still being regular.
Later on, the second and third authors \cite{YY23}
obtained a full characterization of the possible cofinalities of the least $\lambda$-strongly compact cardinal.
In a follow-up paper \cite{YY23b}, they extended these results and gave
a complete picture dealing with \emph{regular} $\lambda$-strongly compact cardinals for multiple $\lambda$'s simultaneously.
In contrast, whether there could exist singular $\lambda$-strongly compact cardinals for different $\lambda$'s simultaneously
remained open.
The difficulty hinges on the fact that the only known way to make the least $\lambda$-strongly compact cardinal a singular cardinal is via \emph{Radin forcing},
whereas Radin forcing over a prescribed cardinal kills the $\omega_1$-strong compactness of all cardinals below it,
thereby making it impossible to get even two singular $\lambda_i$-strongly compact cardinals for $i<2$ via the product of two Radin forcings.
And indeed, the following question remained open:
\begin{q}[Bagaria and Magidor]\label{ques1}
Under suitable large cardinal assumption, can there consistently exist two regular cardinals $\lambda_0<\lambda_1$
and two singular cardinals $\kappa_0<\kappa_1$ such that $\kappa_i$ is the least $\lambda_i$-strongly compact cardinal for all $i<2$?
\end{q}

In this paper we find a new way to make the least $\lambda$-strongly compact cardinal a singular cardinal and use it to
answer the question of Bagaria and Magidor in the affirmative.
\begin{mainthm}\label{thmb} From a suitable large cardinal hypothesis,
it is consistent that for proper class many cardinals $\lambda$, the least $\lambda$-strongly compact cardinal is singular.
\end{mainthm}

Our third example is a weakening of weak compactness and concerns the \emph{$C$-sequence number}
due to Lambie-Hanson and Rinot (see Definition~\ref{c-seq_num_def} below).
For a regular uncountable cardinal $\kappa$,
this cardinal characteristic $\chi(\kappa)$ takes a value in the interval $[0,\kappa]$,
and as a rule of thumb, the smaller it is, the more compactness properties $\kappa$ possesses.
For instance, every finite family of stationary subsets of $E^\kappa_{>\chi(\kappa)}$ reflects simultaneously \cite[Lemma~2.2]{paper35},
and a strong anti-Ramsey coloring $c:[\kappa]^2\rightarrow\chi(\kappa)$ provably exists \cite[Theorem 5.11]{paper35}.
A cardinal $\kappa$ is weakly compact iff $\chi(\kappa)=0$,
and by \cite[\S3]{paper35}, it is consistent for a strongly inaccessible cardinal $\kappa$ to satisfy
$\chi(\kappa)=\delta$ for any prescribed regular $\delta<\kappa$.
In contrast, all known consistent examples of $\chi(\kappa)$ being a singular cardinal satisfy that $\kappa$ is the successor of $\chi(\kappa)$.
This raises the following question:

\begin{q}[{\cite[Question~6.5]{paper35}}]\label{ques2} Suppose that $\chi(\kappa)$ is a singular cardinal.
Must $\kappa$ be the successor of a cardinal of cofinality $\cf(\chi(\kappa))$?
\end{q}

In this paper we provide a consistent counterexample.
\begin{mainthm}\label{thmc} For every weakly compact cardinal $\kappa$, for every infinite cardinal $\delta\le\kappa$,
there exists a cofinality-preserving forcing extension satisfying $\chi(\kappa)=\delta$.
In particular, it is consistent for the $C$-sequence number of an inaccessible cardinal to be singular.
\end{mainthm}

\subsection{Our method}
Though the above three questions concern three different compactness properties,
it turns out that they can be dealt with uniformly through
the method of taking the intersection of a decreasing sequence of generic extensions.
Intersection models and geometric set theory in general were studied recently by Larson and Zapletal \cite[\S4]{MR4249448}
in view of its applications in descriptive set theory. We arrived at studying
intersection models independently
after realizing they provide a common ground for various existing compactness results
and has the potential of materializing completely new configurations,
as is now demonstrated by the three main theorems of this paper.

\subsection{Organization of this paper}

In Section~\ref{sec3},
we provide a few preliminaries and recall when the intersection of a descending sequence of generic extensions is equivalent to the generic extension
given by the corresponding direct limit forcing.

In Sections \ref{secinterlude} and \ref{treeascentpath} we get equipped with some technical tools.
Specifically, in Section~\ref{secinterlude}, a general method for constructing a limit ultrafilter in an intersection model is given,
and in Section~\ref{treeascentpath}, it is shown that trees in an intersection model partially inherit their branches in the intermediate models along the intersection as an ascent path.

In Section~\ref{sec4}, we solve Ketonen's problem and also extend his theorem that every weakly compact cardinal $\kappa$ carrying an indecomposable ultrafilter is Ramsey.
The proof of the latter goes through the \emph{Welch game} and yields that $\kappa$ is a super Ramsey limit of super Ramsey cardinals.
As a bonus, we also obtain a new proof of \cite[Theorem 1.7]{FMZ24}.
The proof of Theorem~\ref{thma} will be found in this section.

In Section~\ref{sec5} we provide all the building blocks needed to obtain a complete picture about the hierarchy
of $\lambda$-strongly compact cardinals for multiple cardinals $\lambda$ simultaneously.
The proof of Theorem~\ref{thmb} will be found in this section.

In Section~\ref{sec6},
we prove that for a weakly compact cardinal $\kappa$,
the $C$-sequence number of $\kappa$ can be forced to become any prescribed infinite cardinal that is no more than $\kappa$.
The proof of Theorem~\ref{thmc} will be found in this section.

\section{Preliminaries}\label{sec3}

Pioneering results on the intersection of a decreasing sequence of models
are due to Sakarovitch \cite{Sak77} and Jech \cite{Jech78}.
Recently, in the study of singular cardinal combinatorics,
the method of taking an intersection of a sequence of models found applications in the work of Hayut and Unger \cite{HU20}
and the work of Ben-Neria, Hayut and Unger \cite{MR4725661}.
Most notably, there are applications
when the intersection model fails to satisfy the Axiom of Choice.
Larson and Zapletal \cite[\S4]{MR4249448} as well as Shani \cite{Shani21} obtained results related to the Borel equivalence relation $E_1$
by analyzing the tail intersection model of $\add(\omega,\omega)$-generic extensions.
In addition, by analyzing the intersection of a decreasing sequence of models,
recently Lietz \cite{Lietz24} proved that the Axiom of Choice may fail in the $\kappa$-mantle, i.e., the intersection of all grounds
that extend to $V$ via a forcing of size less than $\kappa$,
if $\kappa$ is a Mahlo cardinal, $\kappa=\omega_1$ or $\kappa$ is a successor of a regular cardinal.
Here we review some of the results that we will be needing for our applications.

\begin{defn}
Suppose $\mathbb{P}$ and $\mathbb{Q}$ are two posets.
A map $\pi:\mathbb{Q}\rightarrow \mathbb{P}$ is a \emph{projection} iff the following three hold:
\begin{enumerate}
\item $\pi$ is surjective;
\item $\pi$ is order-preserving, that is, $p \leq_{\mathbb{Q}} q$ implies $\pi(p) \leq_{\mathbb{P}} \pi(q)$;
\item For all $p \in \mathbb{P}$ and $q \in \mathbb{Q}$ such that $p \leq_{\mathbb{P}} \pi(q)$
there is a $q'\in \mathbb{Q}$ stronger than $q$
such that $\pi(q')\leq_{\mathbb{P}} p$.
\end{enumerate}
\end{defn}
\begin{remark}
Assuming $H$ is $\mathbb{Q}$-generic,
and letting $G$ be the filter of $\mathbb{P}$ generated by $\pi[H]$,
$G$ is $\mathbb{P}$-generic.

Assuming $G$ is $\mathbb{P}$-generic,
the subposet $\mathbb{Q}/ G$ of $\mathbb Q$ is defined to consist of all $q \in \mathbb Q$ with $\pi(q)\in G$.
\end{remark}
\begin{fact}[see {\cite[Lemma 1.2]{MR2768684}}]\label{dense}
Suppose $\pi:\mathbb{Q}\rightarrow \mathbb{P}$ is a projection between two posets.
If $G \s \mathbb{P}$ is generic and $D \s \mathbb{Q}$ is a dense open subset,
then $D \cap \mathbb{Q}/ G$ is dense open in $\mathbb{Q}/ G$.
\end{fact}

\begin{defn}
A poset $\I$ is \emph{directed} iff all $i,j\in\I$, there is a $k\in\I$ such that $k\ge_{\I}i$ and $k\ge_{\I}j$.
\end{defn}
For every function $f$ with domain $\I$, let $f_{\geq i}$ denote the function $f \restriction \{ j \geq_{\I} i \mid j \in \I\}$
and let $f_{\ngeq i}$ denote $f \setminus f_{\geq i}$.

\begin{defn} A poset $\mathbb{P}$ is \emph{$\delta$-directed closed} iff every directed subset
of $\mathbb{P}$ of size no more than $\delta$ has a lower bound in $\mathbb{P}$.
It is \emph{${<}\delta$-directed closed} iff it is $\gamma$-directed closed for all $\gamma<\delta$.
\end{defn}

\begin{defn} A poset $\mathbb P$ is \emph{$\delta$-closed} iff every decreasing sequence of conditions in $\mathbb P$ of length no more than $\delta$ has a lower bound in $\mathbb P$.
It is \emph{${<}\delta$-closed} iff it is $\gamma$-closed for all $\gamma<\delta$.
\end{defn}

\begin{defn}\label{def27} A poset $\mathbb P$ is \emph{$\delta$-distributive} iff the intersection of no more than $\delta$-many dense open subsets of $\mathbb P$ is nonempty (equivalently, dense open).
It is \emph{${<}\delta$-distributive} iff it is $\gamma$-distributive for all $\gamma<\delta$.
\end{defn}

\begin{fact}[{\cite[Theorem~58]{Jech78}}]
Suppose $\mathbb{P}$ is $\kappa$-distributive and
$G \s \mathbb{P}$ is $V$-generic.
Then $V$ is closed under $\kappa$-sequence in $V[G]$.
\end{fact}

\begin{defn}\label{commutativeprojectionsystem}
Suppose $\I$ is a directed poset.
We say that $\vec{\mathbb{P}}=\langle \mathbb{P}_i,\pi_{i,j} \mid i,j\in\I,~i\leq_{\I} j \rangle$ is a \emph{commutative projection system} iff all of the following hold:
\begin{itemize}
\item for every $i \in \I$, $\mathbb{P}_i$ is a poset;
\item for all $i\le_{\I}j$, the map $\pi_{i,j}:\mathbb{P}_i \rightarrow \mathbb{P}_j$ is a projection
\item for all $i\le_{\I}j\le_{\I}k$, $\pi_{i,k}=\pi_{j,k}\circ \pi_{i,j}$.
\end{itemize}

In the special case that each $\pi_{i,j}$ is the identity map, we write $\vec{\mathbb{P}}$ as $\langle \mathbb{P}_i \mid i \in \I\rangle$.
\end{defn}
We always assume that $\I$ has a minimal element $0$.
By convention, we write $\pi_i:=\pi_{0,i}$ for every $i\in\I$.
Note that if $G_0$ is $\mathbb{P}_{0}$-generic over $V$,
then letting $G_i$ be the filter generated by $\pi_i[G_0]$,
it is the case that $G_i$ is $\mathbb{P}_i$-generic.

\begin{defn}
Given a system $\vec{\mathbb{P}}$ as in Definition~\ref{commutativeprojectionsystem},
we define its direct limit forcing $\mathbb P_\infty$ in two steps, as follows:
\begin{itemize}
\item Define a binary relation $E$ over $\mathbb{P}_0$
by letting $p\mathrel{E}q$ iff there exists an $i\in \I$ such that $\pi_i(p)=\pi_i(q)$.
Since $\I$ is directed, $E$ is an equivalence relation.
\item Let $\mathbb{P}_\infty:=\mathbb{P}_0/ E$,
and for $p,q \in \mathbb{P}_0$,
$[q]_E\leq_{\mathbb{P}_\infty}[p]_E$ iff
there exists an $i \in \I$ such that $\pi_i(q)\leq_{\mathbb{P}_i}\pi_i(p)$.
Since $\I$ is directed, $\leq_{\mathbb{P}_\infty}$ is a partial order.
\end{itemize}
\end{defn}
Clearly, for any $i \in \I$,
we have a reduced projection $\pi_{i,\infty}:\mathbb{P}_i \rightarrow \mathbb{P}_\infty$
defined via $\pi_{i,\infty}(\pi_i(p)):=[p]_E$ for any $p\in \mathbb{P}_0$.
To see this is well-defined, note that if $\pi_i(p)=\pi_i(p')$, then $[p]_E=[p']_E$.
Also note that if $G_0$ is $\mathbb P_0$-generic over $V$, then letting $G_\infty$ be the filter generated by $\pi_{0,\infty}[G_0]$,
it is the case that $G_\infty$ is $\mathbb{P}_\infty$-generic over $V$.

Given a cardinal $\delta$ and
a commutative projection system $\vec{\mathbb{P}}=\langle \mathbb{P}_i,\pi_{i,j} \mid i\leq j<\delta \rangle$.
By \cite{Sak77} (see also \cite[Lemma 26.6]{Jech78} for details),
if $\mathbb{P}_0$ is $\delta$-distributive,
then $\bigcap_{i<\delta}V[G_i]=V[G_\infty]$.

\begin{fact}[essentially \cite{Sak77}]\label{thmintersection}
Suppose $\vec{\mathbb{P}}$ is a system as in Definition~\ref{commutativeprojectionsystem}
and $\I$ is a directed poset of cardinality, say $\delta$,
and that $\mathbb{P}_0$ is $\delta$-distributive.
Let $G_0$ be $\mathbb{P}_{0}$-generic over $V$.
For every $i\in\I$, let $G_i$ be the filter generated by $\pi_i[G_0]$,
and let $G_\infty$ be the filter generated by $\pi_{0,\infty}[G_0]$.

Then $V[G_\infty]=\bigcap_{i\in\I} V[G_i]$. In particular, since the former is a forcing extension, the latter is a model of $\zfc$.
\end{fact}

\begin{defn} Given a system $\vec{\mathbb{P}}$ as in Definition~\ref{commutativeprojectionsystem},
we say $\vec{\mathbb{P}}$ is \emph{eventually trivial} iff
for all $p,q \in \mathbb P_{0}$,
there exists an $i\in \I$ such that $\pi_i(p)$ and $\pi_i(q)$ are compatible.
\end{defn}
Note that if $\vec{\mathbb{P}}$ is eventually trivial,
then $\mathbb{P}_{\infty}$ is a trivial forcing
in the sense that its conditions are pairwise compatible.

Using $G_\infty$, we may derive a reduced eventually trivial commutative projection system $\vec{\mathbb{P}}/ G_{\infty}:=\langle \mathbb{P}_i / G_{\infty}, \bar{\pi}_{i,j} \mid i \in \I \rangle$, where $\bar{\pi}_{i,j}:=\pi_{i,j}\restriction(\mathbb{P}_i / G_{\infty})$ for every pair $i \leq_{\I} j$.

\begin{corollary}\label{corV}
Suppose $\vec{\mathbb{P}}$ is a system as in Definition~\ref{commutativeprojectionsystem}.
If $\mathbb{P}_0$ is $|\I|$-distributive and eventually trivial, Then $V=\bigcap_{i \in \I} V[G_i]$.\qed
\end{corollary}

We conclude this section by pointing out that given an ordinal $\theta$, we can naturally define a product of the commutative projection system which results in a new commutative project system.
Namely, given for each $\alpha<\theta$, a commutative projection system
$\vec{\mathbb{P}}^{\alpha}=\langle\mathbb{P}^{\alpha}_i,\pi^{\alpha}_{i,j} \mid i,j\in\I^\alpha,~i\leq_{\I^\alpha} j \rangle$ and its direct limit forcing $\mathbb{P}_\alpha$,
then letting $\mathbb{\vec{I}}$ be the full product of $\I^\alpha$ with $\alpha<\theta$
and letting $\prod_{\alpha<\theta}\vec{\mathbb{P}}^{\alpha}:=
\langle\prod_{\alpha<\theta}\mathbb{P}^{\alpha}_{f(\alpha)},\prod_{\alpha<\theta}\pi^{\alpha}_{f(\alpha),g(\alpha)} \mid f,g\in\vec{\I},~f\leq_{\vec{\I}} g \rangle$
denote the product of $\vec{\mathbb{P}}^{\alpha}$ with $\alpha<\theta$,
it is the case that $\prod_{\alpha<\theta}\vec{\mathbb{P}}^{\alpha}$ is a commutative projection system and
$\prod_{\alpha<\theta}\mathbb{P}_\alpha$ is its direct limit forcing. This will be used in proving Theorem~\ref{thm6.9} below.

\section{Interlude on limit ultrafilters}\label{secinterlude}
\begin{defn}
Suppose $U$ is an ultrafilter over a directed poset $\I$.
$U$ is \emph{fine} iff $\{j \mid j \geq_{\I} i\} \in U$ for every $i \in \I$.
In the special case that $\I=\langle \kappa, {\leq} \rangle$ for some regular cardinal $\kappa$,
being fine is the same as being \emph{uniform}.
In this case, we say that $U$ is \emph{weakly normal} iff for every $X \in U$ and every regressive function $f:X \rightarrow \kappa$,
there exists a $\xi<\kappa$ such that $\{\alpha<\kappa \mid f(\alpha)<\xi\}\in U$.
\end{defn}
If $U$ is weakly normal, then $[\id]_U=\sup(j_U[\kappa])$,
where $j_U:V \rightarrow M_U \cong \ult(V,U)$ is the corresponding ultrapower map given by $U$.
\begin{defn}[Prikry, \cite{prikrythesis}]\label{indecomposable}
Suppose $U$ is an ultrafilter over a set $I$, and let $\mu$ be an infinite cardinal. $U$ is
said to be \emph{$\mu$-decomposable} if there exists a function $f : I \rightarrow \mu$ such that
$f^{-1}[H] \not\in U$ for every $H \in [\mu]^{<\mu}$.
Otherwise, it is said to be \emph{$\mu$-indecomposable}.
\end{defn}
An ultrafilter is $[\delta,\kappa)$-\emph{indecomposable} if it is $\mu$-indecomposable for every cardinal $\mu\in [\delta, \kappa)$,
and similarly it is $(\delta,\kappa)$-\emph{indecomposable} iff it is $[\delta^+,\kappa)$-\emph{indecomposable}.
An ultrafilter over $\kappa$ is \emph{indecomposable} if it is $[\omega_1, \kappa)$-indecomposable.
Given an ultrafilter $U$ over an infinite cardinal $\kappa$,
for any regular $\mu<\kappa$, $U$ is $\mu$-decomposable
iff $\sup(j_U[\mu])<j_U(\mu)$.

\begin{defn}[{\cite{BM2014,BMO2014}}] \label{dl}
Suppose $\kappa \geq \lambda$ are uncountable cardinals.
\begin{itemize}
\item For every $\alpha \geq \kappa$, $\kappa$ is \emph{$\lambda$-strongly compact up to $\alpha$}
iff there exists a definable elementary embedding $j:V \rightarrow M$ with $M$ transitive, such that
$\mathrm{crit}(j) \geq \lambda$ and there exists a $D \in M$ such that $j``\alpha \subseteq D$ and $M \models |D|<j(\kappa)$.
\item $\kappa$ is \emph{$\lambda$-strongly compact} iff $\kappa$ is $\lambda$-strongly compact up to $\alpha$ for every $\alpha \geq \kappa$.
\end{itemize}

$\kappa$ is strongly compact iff it is $\kappa$-strongly compact.
\end{defn}

The following fact is essentially due to Ketonen (see \cite{Usu20} for details).
\begin{fact}[essentially \cite{MR0469768}]\label{usuba}
Suppose $\kappa \geq \lambda$ are two uncountable cardinals. Then $\kappa$
is $\lambda$-strongly compact iff for every regular $\mu \geq \kappa$, there is a
$\lambda$-complete uniform ultrafilter over $\mu$.
\end{fact}

\begin{defn}
Suppose $\I$ is a directed poset and
$\langle \kappa_i \mid i \in \I \rangle$ is a non-decreasing sequence of infinite cardinals.
Suppose $W$ is a fine ultrafilter over $\I$ and
$U_i$ is a fine ultrafilter over $\kappa_i$ for each $i \in \I$.
Then the limit ultrafilter $W\text{-}\lim_{i \in \I} U_i$ over $\sup_{i \in \I}\kappa_i$ is defined as follows:
\[
W\text{-}\lim_{i \in \I} U_i:=\{X\mid \{i \in \I \mid X \cap \kappa_i \in U_i \} \in W\}.
\]
\end{defn}
Clearly, the limit ultrafilter is fine.

We first give a limit ultrafilter construction through the approach of
intersection of a sequence of outer models.
\begin{thm}\label{menas-sheard}
Suppose $\I$ is a directed poset
and there exists a decreasing sequence of inner models $\vec{M}:=\langle M_i \mid i \in \I \rangle$
with $M:=\bigcap_{i \in \I}M_i$ an inner model,
two weakly increasing sequences $\vec{\mu}:=\langle \mu_i \mid i \in \I \rangle$
and $\vec{\kappa}:=\langle \kappa_i \mid i \in \I \rangle$
and a sequence $\vec{U}:=\langle U_i \mid i \in \I\rangle$ such that
\begin{enumerate}
\item $\I \in M$, $W \in M$ is a fine ultrafilter on $\I$ and $\mathcal{P}(\I)^{M_0}=\mathcal{P}(\I)^M$;
\item\label{definable} $\vec{M}_{\geq i}$, $\vec{\mu}_{\geq i}$,
$\vec{\kappa}_{\geq i}$
and $\vec{U}_{\geq i}$ are definable in $M_i$ for every $i\in \I$;
\item\label{completeness} For every $i \in \I$, in $M_i$, $U_i$ is a fine $\mu_i$-complete filter over $\kappa_i$ and $U_i \cap M$ measures all sets in $\mathcal P(\kappa_i)^M$.
\end{enumerate}
Then the limit ultrafilter $W\text{-}\lim_{i\in \I}U_i$ is in $M$ and is fine. Moreover,
for any regular $\lambda\leq \sup_{i \in \I}\mu_i$,
if $W$ is $\lambda$-complete, then $W\text{-}\lim_{i\in \I}U_i$ is $\lambda$-complete,
and if $M_0 \vDash {}^{|\I|}M \s M$,
then $W\text{-}\lim_{i \in \I}U_i$ is $(|\I|,\sup_{i\in \I}\mu_i)$-indecomposable.
\end{thm}
\begin{proof}
Let $\delta =|\I|^M$, let $\mu:=\sup_{i \in \I}\mu_i$ and let $\kappa:=\sup_{i \in \I}\kappa_i$.
Let $U:=W\text{-}\lim_{i \in \I}U_i$.
Then $U$ is a fine filter over $\kappa$.
Since $\mathcal{P}(\I)^{M_0}=\mathcal{P}(\I)^M$,
$W$ is an ultrafilter, and $U_i \cap M$ measures all sets in $\mathcal P(\kappa_i)^M$ for any $i \in \I$,
we have that for every $X \in \mathcal{P}(\kappa)^M$,
either $\{i \in \I \mid X \cap \kappa_i\in U_i \} \in W$ or
$\{i \in \I \mid \kappa_i \setminus X \in U_i \}\in W$.
Thus $U$ measures every element in $\mathcal{P}(\kappa)^M$.
For every $i \in \I$,
since $W$ is fine,
it follows that $U=W\text{-}\lim_{j\geq_{\I}i}U_j$.
Now by \eqref{definable}, $U=W\text{-}\lim_{j\geq_{\I}i}U_j \in M_i$.
Thus $U \in \bigcap_{i \in \I}M_i=M$.

Suppose $W$ is $\lambda$-complete for some $\lambda\leq \mu$.
Since $\lambda\leq \mu$,
there are co-boundedly many $i\in \I$ such that $\lambda\leq\mu_i$.
Thus $U_i$ is $\lambda$-complete for every such $i$ by \eqref{completeness}.
Then by the $\lambda$-completeness of $W$,
it is obvious that $U$ is $\lambda$-complete.

Suppose $M_0 \vDash {}^{\delta}M \s M$.
To see that $U$ is $(\delta,\mu)$-indecomposable,
take any cardinal $\nu \in [\delta^+,\mu)$
and any $f:\kappa \rightarrow \nu$ in $M$.
For co-boundedly many $i \in \I$ (more precisely, any $i \in \I$ with $\mu_i>\nu$),
since $U_i$ is a $\mu_i$-complete fine ultrafilter in $M_i$,
there is a $\xi_i<\mu$ such that $f^{-1}\{\xi_i\}\in U_i$.
Let $A$ be the collection of these $\xi_i$.
Then $|A|\leq \delta$. Thus $A \in M$ since $M_0 \vDash {}^{\delta}M \s M$.
Moreover, $f^{-1}[A]\in U_i$ for co-boundedly many $i \in \I$.
So $f^{-1}[A] \in U$ since $W$ is fine.
\end{proof}

\begin{remark}\label{rmk37}
Suppose $\kappa$ is a supercompact cardinal
and $\delta<\kappa$ is a regular cardinal.
Via the Bukovský-Dehornoy phenomenon,
we can get a suitable iterated ultrapower $\langle M_i, \pi_{i,j} \mid i\leq j\leq \delta \rangle$
such that $M:=\bigcap_{i<\delta}M_i[P \restriction i]=M_\delta[P]$ for some
$\delta$-sequence $P$ and $M_0 \vDash {}^{\delta}M \s M$ (see \cite[Theorem 67]{hayut2023prikry} for details).
Consider the linearly ordered set $\I:=(\{ j+1\mid j<\delta\},{\in})$.
Then for every $j<\delta$, $\pi_{0,j+1}(\kappa)$ is supercompact in $M_{j+1}[P \restriction j+1]$ since
$\pi_{0,j+1}(\kappa)$ is supercompact in $M_{j+1}$ and $M_{j+1}[P\restriction j+1]=M_{j+1}[P\restriction j]$
is a generic extension of $M_{j+1}$ of some small forcing of size less than $\pi_{0,j+1}(\kappa)$ .
Thus if $\delta$ carries a $\lambda$-complete uniform ultrafilter for some $\lambda \leq \delta$,
then by letting $\mu_i=\kappa_i$ for each $i\in\I$,
we can prove $\pi_{0,\delta}(\kappa)$ is $\lambda$-strongly compact in $M$ by Theorem~\ref{menas-sheard}.
This provides a different approach to the results of \cite[Theorem 6.1]{BM2014} and \cite[Theorem 4.2]{YY23}.
\end{remark}

\begin{corollary}
Suppose $\langle \kappa_i \mid i<\delta \rangle$ and $\langle \lambda_i \mid i<\delta \rangle$
are two increasing sequences with $\delta=\sup_{i<\delta}\lambda_i$.
If $\kappa_i$ is $\lambda_i$-strongly compact for each $i<\delta$
and $\delta$ carries a $\lambda$-complete uniform ultrafilter for a given $\lambda \leq \delta$,
then $\sup_{i<\delta}\kappa_i$ is $\lambda$-strongly compact.\qed
\end{corollary}
\begin{remark}
The preceding extends Menas' result \cite{MR357121} that every measurable limit of strongly compact cardinals is strongly compact.
\end{remark}

By combining Fact~\ref{thmintersection} and Theorem~\ref{menas-sheard},
we have the following corollary:
\begin{corollary}\label{maincor1}
Suppose:
\begin{itemize}
\item $\delta$ is an infinite regular cardinal;
\item $\langle \kappa_i \mid i<\delta \rangle$ is a non-decreasing sequence of cardinals converging to some cardinal $\kappa$,
with $\kappa_0>\delta$;
\item $\vec{\mathbb P}:=\langle \mathbb{P}_i, \pi_{i,j} \mid i\leq j<\delta \rangle$ is a commutative projection system;
\item $\mathbb{P}_0$ is $\delta$-distributive.
\end{itemize}
Let $\mathbb{P}_\delta$ be the direct limit forcing of $\vec{\mathbb P}$.
Let $G_0$ be $\mathbb{P}_{0}$-generic over $V$.
For every $i\leq \delta$, let $G_i$ be the filter generated by $\pi_{0,i}[G_0]$.
Then the following holds in $V[G_\delta]$:
\begin{enumerate}
\item\label{maincor2} If $\kappa_i$ is measurable in $V[G_i]$ for every $i<\delta$,
then there is a $(\delta,\kappa)$-indecomposable ultrafilter over $\kappa$;
\item\label{maincor3} If $\delta$ carries a $\lambda$-complete uniform ultrafilter,
and $\kappa_i$ is $\lambda$-strongly compact in $V[G_i]$ for each $i<\delta$,
then $\kappa$ is $\lambda$-strongly compact.\qed
\end{enumerate}
\end{corollary}
The preceding gives an abstract proof for the main theorem of \cite{Sheard83}, \cite[Theorem~3.1]{Git20}
and \cite[Theorems 4.24 and 4.26]{HRZ24}.

\section{Interlude on trees and their ascent paths}\label{treeascentpath}
A set $T$ is a \emph{binary tree} iff it is a subset of ${}^{<\gamma}2$ for some ordinal $\gamma$
and satisfies that for all $t\in T$ and $\alpha < \dom(t)$, $t \restriction \alpha \in T$.
Throughout, we identify $T$ with the poset $\mathbf T:=(T,{\s})$ which is a tree in the abstract set-theoretic sense
whose $\alpha^{\text{th}}$ level is nothing but $T_\alpha:=\{ t\in T \mid \dom(t)=\alpha \}$.
The \emph{height} of $T$, denoted $\h(T)$, stands for the least ordinal $\gamma$ such that $T\s{}^{<\gamma}2$.
The tree $T$ is \emph{normal} iff for all $\alpha<\beta<\h(T)$,
for every $x\in T_\alpha$, there exists a $y\in T_\beta$ with $x\s y$.
To streamline the matter, here a \emph{$\kappa$-tree} is a binary tree $T$ satisfying $\h(T)=\kappa$ and $|T_\alpha|<\kappa$ for all $\alpha<\kappa$.
A \emph{$\kappa$-Aronszajn tree} (resp.~\emph{$\kappa$-Souslin tree}) is a $\kappa$-tree with no chains (resp. chains or antichains) of size $\kappa$.
An uncountable cardinal $\kappa$ is \emph{weakly compact} iff it is strongly inaccessible and there are no $\kappa$-Aronszajn trees.

\begin{defn} A binary tree $T$ is
\emph{uniformly homogeneous} iff for all $\alpha<\beta<\h(T)$,
$x\in T_\alpha$ and $y\in T_\beta$, we have that $x*y$ is in $T$. Here, $x*y$ is the unique function from $\beta$ to $2$ to satisfy:
$$(x*y)(\iota):=\begin{cases}x(\iota),&\text{if }\iota<\alpha;\\
y(\iota),&\text{otherwise.}\end{cases}$$
\end{defn}

\begin{defn}\label{defn42}
Suppose that $T$ is a $\kappa$-tree for a given infinite cardinal $\kappa$,
and that $D$ is a fine filter over some directed poset $\I$.
\begin{enumerate}
\item A sequence $\vec{f}=\langle f_{\beta} \mid \beta<\kappa\rangle$ is a \emph{$D$-ascent path} through $T$ iff the two hold:
\begin{itemize}
\item for every $\beta<\kappa$, $f_{\beta}$ is a function from $\I$ to $T_\beta$;
\item for all $\alpha<\beta<\kappa$, the set $\{i \in \I \mid f_{\alpha}(i)\s f_{\beta}(i)\}$ is in $D$.
\end{itemize}
\item A pair $(Y,B)$ is a \emph{thread} through a $D$-ascent path $\vec{f}=\langle f_{\beta} \mid \beta<\kappa\rangle$ iff all of the following hold:
\begin{itemize}
\item $Y \in D$;
\item $B$ is a cofinal subset of $\kappa$;
\item for every pair $\alpha<\beta$ of ordinals from $B$, $\{ i\in\I\mid f_{\alpha}(i)\s f_\beta(i)\}$
covers $Y$;
\item(maximality) for every $\alpha<\kappa$, if there exists a $\beta \in B$ above $\alpha$ such that
$\{ i\in\I\mid f_{\alpha}(i)\s f_\beta(i)\}$ covers $Y$, then $\alpha \in B$.
\end{itemize}
\end{enumerate}
\end{defn}
Laver's original definition of a \emph{$\delta$-ascent path} corresponds to the special case in which $\I=(\delta,{\in})$ and $D$ is the filter of co-bounded subsets of $\delta$.

\begin{fact}[{\cite[Lemmas 3.7 and 3.38(3)]{MR4636632}}]\label{killstrongcompactness}
Suppose $\delta<\kappa$ is a pair of infinite regular cardinals.
If there exists a $\kappa$-Aronszajn tree with a $\delta$-ascent path,
then every uniform ultrafilter over $\kappa$ is $\delta$-decomposable.
\end{fact}

Coming back to the general case of Definition~\ref{defn42},
and given a thread $(Y,B)$ through some $D$-ascent path $\vec{f}=\langle f_\beta\mid \beta<\kappa \rangle$,
for every $X\s Y$ in $D$,
we let $B^X:=\{\alpha<\kappa \mid \forall i \in X\,[f_{\alpha}(i)\s f_{\min(B \setminus \alpha)}(i)]\}$.
Clearly, $(X, B^X)$ is a thread through $\vec{f}$.
The following lemma may be extracted from the proof of \cite[Theorems~5.1]{lh_lucke}.

\begin{lemma}[Lambie-Hanson and L\"{u}cke, \cite{lh_lucke}]\label{rigid}
Suppose that $T$ is a $\kappa$-tree for a given infinite cardinal $\kappa$,
and that $D$ is a fine filter over some directed poset $\I$.
Suppose also:
\begin{itemize}
\item $\vec{f}=\langle f_\beta\mid \beta<\kappa \rangle$ is a $D$-ascent path through $T$;
\item for every $\epsilon<2$, $(Y_\epsilon,B_\epsilon)$ is a thread through $\vec f$;
\item $\cf(\kappa) > |\I|$.
\end{itemize}
Then there exists a subset $X\s Y_0\cap Y_1$ in $D$ such that $B^X_0=B^X_1$.
\end{lemma}
\begin{proof} As $\vec f$ is a $D$-ascent path, we may define a function $g:B_1 \rightarrow D$ via
$$g(\beta):=\{i \in Y_0 \cap Y_1 \mid f_{\beta}(i)\s f_{\min(B_0 \setminus \beta)}(i)\}.$$
\begin{claim} Let $\alpha<\beta$ be a pair of ordinals in $B_1$.
Then $g(\beta)\s g(\alpha)$.
\end{claim}
\begin{proof}
Let $i \in g(\beta)$.
Since $(Y_0,B_0)$ is a thread through $\vec{f}$ and since $i \in Y_0$,
we have $$f_{\min(B_0 \setminus \alpha)}(i)\s f_{\min(B_0 \setminus \beta)}(i).$$
In addition, $(Y_1,B_1)$ is a thread through $\vec{f}$, so since $i\in g(\beta)\s Y_1$,
we have that $$f_{\alpha}(i)\s f_{\beta}(i)\s f_{\min(B_0 \setminus \beta)}(i).$$
As $T$ is a tree, it follows that $f_{\alpha}(i)\s f_{\min(B_0 \setminus \alpha)}(i)$,
meaning that $i\in g(\alpha)$.
\end{proof}

As $\cf(\kappa) > |\I|$, it now follows that $g$ is eventually constant with value, say, $X$.
Then $B^X_0$ coincides with $B^X_1$ by the maximality of the two.
\end{proof}

\begin{corollary}
Suppose $V,W_0,W_1,W$ are transitive models of $\zfc$, with
$V\s W_0,W_1 \s W$.
Suppose also:
\begin{itemize}
\item $\kappa, T, \I$ and $D$ are given as above in $V$;
\item $\vec{f}=\langle f_\alpha \mid \alpha<\kappa \rangle$ is a $D$-ascent path through $T$ in $V$;
\item for every $\epsilon<2$, $(Y_\epsilon,B_\epsilon)$ is a thread through $\vec f$ in $W_\epsilon$;
\item $\cf^{W}(\kappa)>|\I|$.
\end{itemize}
Then there is a thread through $\vec{f}$ in $W_0 \cap W_1$.\qed
\end{corollary}

\begin{corollary}\label{twothreads}
Suppose $\delta<\kappa$ is a pair of infinite regular cardinals
and $\vec{f}=\langle f_\alpha \mid \alpha<\kappa \rangle$ is
a $\delta$-ascent path through some $\kappa$-Aronszajn tree.
For every poset $\mathbb{P}$, if $\mathbb{P}\times \mathbb{P}$
is $\delta$-distributive,
then $\mathbb{P}$ does not add a thread through $\vec{f}$.\qed
\end{corollary}
We now give an abstract proof of \cite[Theorems 1.3 and 5.1(3)]{lh_lucke}.
\begin{corollary}
Suppose:
\begin{itemize}
\item $\delta<\kappa$ are two infinite regular cardinals;
\item $\vec{\mathbb P}=\langle \mathbb{P}_i, \pi_{i,j} \mid i\le j<\delta \rangle$
is an eventually trivial commutative projection system such that
\begin{itemize}
\item $\mathbb{P}_0$ is $\delta$-distributive, and
\item $\mathbb{P}_i$ forces that $\kappa$ is weakly compact for every $i<\delta$.
\end{itemize}
\end{itemize}
Then:
\begin{enumerate}
\item\label{ascentpath} every $\kappa$-Aronszajn tree has a $\delta$-ascent path;
\item\label{noascentpath} for every regular cardinal $\theta<\kappa$ distinct from $\delta$,
there is no $\kappa$-Aronszajn tree with a $\theta$-ascent path.
\end{enumerate}
\end{corollary}
\begin{proof}
Let $G_0 \s \mathbb{P}_{0}$ be generic over $V$,
and let $G_i$ be the filter generated by $\pi_i[G_0]$.
Then $G_i$ is $\mathbb{P}_i$-generic.
Since $\vec{\mathbb P}$ is an eventually trivial commutative projection system
with $\mathbb{P}_0$ a $\delta$-distributive poset,
it follows from Corollary~\ref{corV} that $\bigcap_{i<\delta}V[G_i]=V$.

\eqref{ascentpath}
Take any $\kappa$-Aronszajn tree $T$ in $V$.
For every function $f:\delta \rightarrow T$,
let $[f]_E:=\{g \in {}^{\delta}T \mid \exists i<\delta (g \restriction [i,\delta)=f\restriction [i,\delta))\}$.

For every $i<\delta$, since $\mathbb{P}_i$ forces that $\kappa$ is weakly compact,
we may fix a cofinal branch $b_i$ through $T$ in $V[G_i]$.
Now consider $\langle f_\beta \mid \beta<\kappa \rangle$,
where $f_\beta(i):=b_i(\beta)$ for all $\beta<\kappa$ and $i<\delta$.
Then $\langle f_\beta\restriction [i,\delta) \mid \beta<\kappa \rangle \in V[G_i]$ for every $i<\delta$.
Since $\mathbb{P}_0$ is $\delta$-distributive,
it follows that $[f_\beta]_{E} \in \bigcap_{i<\delta} V[G_i]=V$ for every $\beta<\kappa$.
Thus $\langle [f_\beta]_{E} \mid \beta<\kappa \rangle \in \bigcap_{i<\delta} V[G_i]=V$.
This witnesses that $T$ admits a $\delta$-ascent path.

\eqref{noascentpath}
Suppose $T$ is a $\kappa$-Aronszajn tree with a $\theta$-ascent path $\vec{f}=\langle f_\beta \mid \beta<\kappa \rangle$,
for some regular $\theta\neq\delta$ below $\kappa$.
Recall that this means that we work with the directed set $\I:=(\theta,{\in})$ and the filter $D$ of co-bounded subsets of $\theta$.
For every $i<\delta$, in $V[G_i]$, $\kappa$ is weakly compact,
so we may fix there a thread $(Y_i,B_i)$ through $\vec f$.
By Lemma \ref{rigid}, for every $i<\delta$,
there exists an $X_i\s Y_0\cap Y_i$ in $D$ such that $B_i^{X_i}=B_0^{X_i}$.
Recalling the nature of our filter $D$, we may find some $\tau_i<\theta$ such that $X_i=\theta\setminus\tau_i$.

As $\theta$ and $\delta$ are two distinct regular cardinals, we may fix some $\tau<\theta$
for which $I:=\{i<\delta\mid \tau_i\le\tau\}$ has size $\delta$.
Thus letting $X:=\theta\setminus\tau$, the thread $(X,B_0^{X})$ is in $\bigcap_{i \in I}V[G_i]$.
But the latter coincides with our universe $V$, contradicting the fact that $T$ is Aronszajn.
\end{proof}
\section{Ketonen's question}\label{sec4}
The main result of this section is Theorem~\ref{thm46}. It implies Theorem~A, answering Ketonen's question in the negative, as follows.

\begin{corollary}\label{coroallaryA}
Assuming the consistency of a measurable cardinal, it is consistent that a weakly compact cardinal $\kappa$ carries a uniform indecomposable ultrafilter, yet $\kappa$ is not measurable.
\end{corollary}
\begin{proof}
Assuming the consistency of the existence of a measurable cardinal,
we may pass to the inner model $\mathrm{L}[U]$ for some normal measure $U$ over a (measurable) cardinal $\kappa$.
Now, appeal to Theorem~\ref{thm46} below with $\delta:=\omega$.
\end{proof}

An interesting feature of the preceding model is that, by \cite[Theorem~3.4(2)]{MR585552},
once $\kappa$ is weakly compact, every uniform indecomposable ultrafilter $D$ over $\kappa$ is Rowbottom.
So, by \cite[Theorem~1]{MR347606} and even though $\kappa$ is not measurable, the corresponding Prikry forcing $\mathbb P_D$
to change the cofinality of $\kappa$ to $\omega$ will not collapse cardinals.

\begin{defn}[Baumgartner-Harrington-Kleinberg, \cite{bhk}] For a regular uncountable cardinal $\kappa$ and a set $F\s\kappa$,
$\club(F)$ stands for the poset consisting of all closed bounded subsets $c$ of $\kappa$ such that $c\s F$.
A condition $c'$ \emph{extends} a condition $c$
iff $c'\cap(\max(c)+1)=c$.
\end{defn}

The next lemma is standard (see for instance \cite[Theorem~1]{MR716625}).
Due to the importance of Clause~(2) in our context, we do provide a proof.
\begin{lemma}\label{preservestationary}
Suppose $\kappa^{<\kappa}=\kappa$, $\delta\leq \kappa$ is an infinite regular cardinal
and that $\vec{S}=\langle S_\gamma\mid \gamma<\delta \rangle$ is a given pairwise disjoint sequence of stationary subsets of $E^{\kappa^+}_\kappa$.
Then
\begin{enumerate}
\item $\prod_{\gamma<\delta}\club(\kappa^+ \setminus S_{1+\gamma})$
is ${<}\kappa$-closed and $\kappa$-distributive;
\item $\prod_{\gamma<\delta}\club(\kappa^+ \setminus S_{1+\gamma})$
preserves the stationarity of every stationary subset of $E^{\kappa^+}_\kappa \setminus \bigcup_{\gamma<\delta}S_{1+\gamma}$.
\end{enumerate}
\end{lemma}
\begin{proof} For every subset $S\s E^{\kappa^+}_\kappa$, as $\kappa^+ \setminus S$ covers $E^{\kappa^+}_{<\kappa}$,
it is the case that $\club(\kappa^+ \setminus S)$ is ${<}\kappa$-closed.
Consequently, the full product $\mathbb P:=\prod_{\gamma<\delta}\club(\kappa^+ \setminus S_{1+\gamma})$ is ${<}\kappa$-closed.

Next, to prove that $\mathbb P$ is $\kappa$-distributive while also proving Clause~(2), we let $T$ be an arbitrary stationary subset of $E^{\kappa^+}_\kappa \setminus \bigcup_{\gamma<\delta}S_{1+\gamma}$.
Note that such a $T$ exists, e.g., $T=S_0$.
Let $\vec{D}=\langle D_i \mid i<\kappa \rangle$ be any sequence of dense open subsets of $\mathbb P$,
let $\dot E$ be a $\mathbb P$-name for some club $E$ in $\kappa^+$,
and let $\vec{c}$ be an arbitrary condition in $\mathbb P$.
We shall find an extension of $\vec c$ lying in $\bigcap_{i<\kappa}D_i$ and forcing that $T$ has a nonempty intersection with $E$.

To this end, fix a sufficiently large regular cardinal $\chi$ and a well-ordering $\lhd_\chi$ of $H_\chi$.
As $\kappa^{<\kappa}=\kappa$ and as $T$ is a stationary subset of $E^{\kappa^+}_\kappa$, we may fix an elementary substructure $N \prec (H_\chi, \in,\lhd_\chi)$ such that
\begin{enumerate}
\item $\vec{c}$, $\mathbb{P}$, $\vec{S}$, $\vec{D}\in N$;
\item $|N|=\kappa$;
\item ${}^{<\kappa}N \s N$;
\item $\theta:=N\cap \kappa^+$ is in $T$.
\end{enumerate}

Let $\langle \theta_i \mid i<\kappa \rangle$ be some increasing and cofinal sequence in $\theta$.
Next, build a sequence of conditions $\langle \vec{c}_i \mid i\leq \kappa \rangle$
as follows:
\begin{itemize}
\item $\vec{c}_0=\vec{c}$;
\item for every $i<\kappa$, $\vec{c}_{i+1}$ is the $\lhd_\chi$-least extension of $\vec c_i$ to satisfy all of the following:
\begin{itemize}
\item $\vec{c}_{i+1}$ is in $D_i$;
\item $\vec{c}_{i+1}$ decides a value $\epsilon_i$ for the first element of $E\setminus \theta_i$;
\item $\max(\vec{c}_{i+1}(\gamma))>\max\{\max(\vec{c}_{i}(\gamma)),\epsilon_i\}$ for all $\gamma<\delta$.
\end{itemize}
\item for every $i\in\acc(\kappa+1)$, for every $\gamma<\delta$, $\vec{c}_i(\gamma)=(\bigcup_{j<i}\vec{c}_j(\gamma))\cup\{\sup(\bigcup_{j<i}\vec{c}_j(\gamma))\}$.
\end{itemize}

It can be verified that $\{ \vec c_i\mid i<\kappa\}\s N$,
so that $\max(\vec{c}_\kappa(\alpha))=\theta$ for all $\alpha<\delta$.
It follows that $\vec{c}_\kappa$ is a legitimate condition lying in $\bigcap_{i<\kappa}D_i$
and forcing that $\theta$ belongs to $E$.
\end{proof}

In the presence of an ultrapower embedding $j:V\rightarrow M$ with critical point $\kappa$,
the preceding lemma motivates the need for a partition $\vec S=\langle S_\gamma\mid\gamma<\delta\rangle$ of $E^{\kappa^+}_\kappa$
into $V$-stationary sets such that $\vec S$ lies in $M$.
This is a well-known challenge. For instance, in \cite{MR2548481},
a two-universe partition of length $\kappa^+$ is obtained from a partial diamond sequence for $\mathrm{L}_{\kappa^+}[j(U)]$ constructed using a bit of fine structure.
Our next result provides a partition of length $\kappa$ from a weak hypothesis that is readily available in our context.

\begin{proposition}\label{stationarysplitting} Suppose that $M\s V$ is a transitive model of $\zfc$,
and $\mathcal P^M(\kappa)=\mathcal P^V(\kappa)$.
For every $S\in\mathcal P^M(\kappa^+)$ such that $V\models S\text{ is stationary}$,
there is a partition $\langle S_\gamma\mid \gamma<\kappa\rangle\in M$ such that $S_\gamma$ is stationary in $V$ for each $\gamma<\kappa$.
\end{proposition}
\begin{proof} In $M$, fix an \emph{Ulam matrix} for $\kappa^+$,
$\vec U=\langle U_{\eta,\tau}\mid \eta<\kappa, \tau<\kappa^+\rangle$. This means that the following two hold:
\begin{enumerate}
\item For every $\eta<\kappa$, $\langle U_{\eta, \tau}\mid \tau<\kappa^+\rangle$ consists of pairwise disjoint subsets of $\kappa^+$;
\item For every $\tau<\kappa^+$, $|\kappa^+\setminus \bigcup_{\eta<\kappa}U_{\eta,\tau}|\le\kappa$.
\end{enumerate}

From $\mathcal P^M(\kappa)=\mathcal P^V(\kappa)$,
we infer that $(\kappa^+ )^M =(\kappa^+)^V$, so that
$V\models \vec U\text{ is an Ulam matrix for }\kappa^+$.
Now, let $S\in\mathcal P^M(\kappa^+)$ be such that $V\models S\text{ is stationary}$.
\begin{claim} There exists an $\eta<\kappa$ such that
$$V\models\text{the set }T:=\{\tau<\kappa^+\mid S\cap U_{\eta,\tau}\text{ is stationary}\}\text{ has size }\kappa^+.$$
\end{claim}
\begin{proof} Work in $V$. By Clause~(2), for every $\tau<\kappa^+$,
$S\setminus \bigcup_{\eta<\kappa}U_{\eta,\tau}$ is nonstationary.
It follows that we may define a map $f:\kappa^+\rightarrow\kappa$ via
$$f(\tau):=\min\{\eta<\kappa\mid S\cap U_{\eta,\tau}\text{ is stationary}\}.$$
Clearly, there exists an $\eta<\kappa$ for which $T:=f^{-1}\{\eta\}$ has size $\kappa^+$.
\end{proof}

Work in $V$.
Let $\eta$ and $T$ be given by the claim, and let $\pi:\kappa^+\rightarrow T$ denote the inverse collapsing map.
As $\mathcal P^M(\kappa)=\mathcal P^V(\kappa)$, it is the case that $\pi\restriction\kappa$ is in $M$.
Recalling Clause~(1), altogether, $\langle S\cap U_{\eta,\pi(i)}\mid i<\kappa\rangle$ is a partition in $M$
of a subset of $S$ into $V$-stationary sets, and it can easily be extended to a partition of the whole of $S$.
\end{proof}

\begin{defn}
Given any ideal $\mathcal{I}$ over $\kappa$, let $\mathcal{P}(\kappa)/ \mathcal{I}$
be the collection of all $\mathcal{I}$-positive sets,
and for any $A,B \in \mathcal{P}(\kappa)/ \mathcal{I}$,
$A \leq_{\mathcal{P}(\kappa)/ \mathcal{I}} B$ iff $B \setminus A \in \mathcal{I}$.
\end{defn}

We now arrive at the main result of this section.

\begin{thm}\label{thm46} Suppose $\kappa$ is a measurable cardinal and that
$V=\mathrm{L}[U]$ for some normal measure $U$ over $\kappa$.
Let $\delta\leq\kappa$ be an infinite regular cardinal.
Then in some forcing extension $V'$, $\kappa$ is a weakly compact cardinal carrying a $(\delta,\kappa)$-indecomposable ultrafilter.
Moreover, all of the following hold:
\begin{enumerate}
\item\label{nosaturated} in no $\kappa^+$-c.c. forcing extension or $\kappa$-closed forcing extension, there is a non-trivial elementary embedding $k':V'\rightarrow M'$ for some transitive class $M'$.
In particular, $\kappa$ is non-measurable, and there is no uniform saturated ideal over $\kappa$;
\item\label{precipitous} there is a uniform normal precipitous ideal $\mathcal{I}$ over $\kappa$ such that $\mathcal{P}(\kappa)/ \mathcal{I}$ has a $\kappa$-distributive ${<}\delta$-closed subset.
\end{enumerate}
\end{thm}
\begin{proof} Recall that by a theorem of Silver \cite{MR278937}, $\gch$ holds in $V$.
Let $j:V\rightarrow M$ be the corresponding ultrapower map given by $U$ with $M$ transitive.
By Proposition~\ref{stationarysplitting}, we may pick in $M$ a partition $\vec{S}=\langle S_\gamma\mid \gamma<\delta \rangle$ of $E^{\kappa^+}_\kappa$
such that each $S_\gamma$ is stationary in $V$.

Let $f:\kappa \rightarrow V_\kappa$ be a representation
of $\vec{S}$ so that, for every inaccessible $\xi<\kappa$, $f(\xi)$ is a $\delta$-sequence ($\xi$-sequence if $\delta=\kappa$)
of pairwise disjoint stationary subsets of $E^{\xi^+}_\xi$.
As $[f]_U=\vec S$ and to streamline the upcoming definition, we shall denote $f(\kappa):=\vec{S}$.
With this convention, we define an Easton support iteration of length $\kappa+1$,
$$(\langle \mathbb P_\xi\mid \xi\le\kappa+1\rangle,\langle \dot{\mathbb{Q}}_{\xi} \mid \xi <\kappa+1\rangle),$$
where for every inaccessible $\xi<\kappa+1$,
$\mathbb{\dot{Q}}_{\xi}$ is a $\mathbb{P}_{\xi}$-name for
the lottery sum $\bigoplus_{{\beta}< \dom(f(\xi))}\mathbb R_\xi^{\beta}$ of the following building blocks
$$\mathbb R_\xi^{\beta}:=\prod_{{\beta} \leq \gamma<\dom(f(\xi))}\club(\xi^+\setminus f(\xi)(1+\gamma)),$$
and for all other $\xi$'s, $\mathbb{\dot{Q}}_{\xi}$ is nothing but a $\mathbb{P}_{\xi}$-name for a trivial forcing.

Let $G_{\kappa}$ be a $\mathbb{P}_{\kappa}$-generic filter over $V$,
and work in $V[G_\kappa]$.
For all ${\beta}\leq \gamma<\delta$,
let $\pi_{{\beta},\gamma}:\mathbb{R}^{\beta}_{\kappa} \rightarrow \mathbb{R}^\gamma_{\kappa}$ be the projection
defined via $\pi_{{\beta},\gamma}(p):=p \restriction [\gamma,\delta)$.
Consider the outcome commutative projection system $\vec{\mathbb R}:=\langle \mathbb{R}^{\beta}_{\kappa}, \pi_{{\beta},\gamma} \mid {\beta}\leq \gamma<\delta \rangle$.
Let $\mathbb{R}^{\delta}_{\kappa}$ be its direct limit forcing.
Let $g_0$ be an $\mathbb{R}^0_{\kappa}$-generic filter over $V[G_{\kappa}]$.
For each $\gamma\leq \delta$, denote $g_\gamma:= \pi_{0,\gamma}[g_0]$, so that $g_\gamma$ is $\mathbb{R}^\gamma_{\kappa}$-generic over $V[G_{\kappa}]$.
Let $V':=\bigcap_{\gamma<\delta}V[G_\kappa][g_\gamma]$.
By Fact~\ref{thmintersection} and
since $\mathbb{R}^0_{\kappa}$ is $\delta$-distributive in $V[G_\kappa]$,
we have that $V'=V[G_\kappa][g_\delta]$.
Thus $V'$ is a model of $\zfc$. We next show $V'$ satisfies our requirements.

\begin{claim}\label{measurable} Let $\gamma<\delta$.
Then $\kappa$ is measurable in $V[G_\kappa][g_\gamma]$.
\end{claim}
\begin{proof}
We first show that $2^\kappa=\kappa^+$ holds in $V[G_\kappa][g_\gamma]$.
Since $\mathbb{P}_\kappa$ satisfies the $\kappa$-c.c. and $|\mathbb{P}_\kappa|=\kappa$,
it follows that there are only $\kappa^{<\kappa}=\kappa$ many antichains.
Thus $2^\kappa=\kappa^+$ in $V[G_\kappa]$.
Since $\mathbb{R}^{\beta}_{\kappa}$ is $\kappa$-distributive for every $\beta<\delta$,
it adds no subsets of $\kappa$.
So $2^\kappa=\kappa^+$ in $V[G_\kappa][g_\gamma]$.

Since $\mathbb{P}_\kappa$ satisfies the $\kappa$-c.c.,
it follows that in $V[G_\kappa]$,
$\kappa$ is a regular cardinal
and $M[G_\kappa]$ is closed under $\kappa$-sequences.
Note also that $\mathbb{R}^{\beta}_{\kappa}$ is $\kappa$-distributive,
it follows that in $V[G_\kappa][g_\gamma]$,
$\kappa$ is a regular cardinal
and $M[G_\kappa][g_\gamma]$ is closed under $\kappa$-sequences.

Since $j(\vec{f})(\kappa)=\vec{S}$, $(\kappa^+)^{M[G_\kappa]}=\kappa^+$,
and $M[G_\kappa]$ is closed under $\kappa$-sequences in $V[G_\kappa]$,
it follows that $(j(\mathbb P_\kappa)/ G_\kappa) (\kappa)=(\mathbb{Q}_\kappa)^{M[G_\kappa]}=\mathbb{Q}_\kappa$.
Since $g_\gamma$ is $\mathbb R_\kappa^\gamma$-generic,
it is in particular $\mathbb Q_\kappa$-generic, so $g_\gamma$ is $(j(\mathbb P_\kappa)/ G_\kappa) (\kappa)$-generic over $M[G_\kappa]$.

Now consider the forcing $j(\mathbb P_\kappa)/(G_\kappa*g_\gamma)$.
By elementarity, it is $\kappa$-closed and the size of all maximal antichains is $j(\kappa)$ in $M[G_\kappa][g_\gamma]$.
Now work in in $V[G_\kappa][g_\gamma]$.
Since $M[G_\kappa][g_\gamma]$ is closed under $\kappa$-sequences and $|j(\kappa)|=2^\kappa=\kappa^+$,
it follows that $j(\mathbb P_\kappa)/(G_\kappa*g_\gamma)$ is $\kappa$-closed and
the size of all maximal antichains in $M[G_\kappa]$ is $\kappa^+$.
So by a standard diagonalization argument,
we can build a $j(\mathbb P)/(G_\kappa*g_\gamma)$-generic, say $G_{tail}$, over $M[G][g_\gamma]$.
Thus we may lift $j$ to an elementary embedding $j^+:V[G_\kappa] \rightarrow M[G][g_\gamma][G_{tail}]$
by Silver's criterion.
Since $\mathbb{R}^\gamma_{\kappa}$ is $\kappa$-distributive,
we may lift $j^+$ to an elementary embedding $j^{++}:V[G_\kappa][g_\gamma] \rightarrow M[G][g_\gamma][G_{tail}][g_\gamma']$
by the transfer argument (see \cite[Remark~15.2]{MR2768691}).
Here $g_\gamma'$ is the filter generated by $j^+[g_\gamma]$.
Thus $j^{++}$ witnesses that $\kappa$ is measurable in $V[G][g_\gamma]$.
\end{proof}

It now follows from Corollary~\ref{maincor1}(1) that $\kappa$
carries a uniform $(\delta,\kappa)$-indecomposable ultrafilter in $V'$.

\begin{claim} $\kappa$ is weakly compact in $V'$.
\end{claim}
\begin{proof} Suppose that $T$ is a $\kappa$-tree in $V'$. In particular, it is a $\kappa$-tree in $V[G_\kappa][g_0]$.
So by the previous claim, $T$ admits a $\kappa$-branch in $V[G_\kappa][g_0]$.
However, $\mathbb{R}^0_{\kappa}$ is $\kappa$-distributive (in $V[G_\kappa]$),
and hence the said branch lies in $V[G_\kappa]$. In particular, $T$ admits a $\kappa$-branch in $V'$.
\end{proof}

Now we prove Clause~\eqref{nosaturated}.
Towards a contradiction, suppose
that $V^*$ is some $\kappa^{+}$-c.c. forcing extension or a $\kappa$-closed forcing extension of $V'$ in which there exists a nontrivial elementary embedding $k':V' \rightarrow M'$.
Consider $k:=k'\restriction V$ which is an elementary embedding from $V$ to some transitive class $N$,
with critical point $\kappa$.
Using \cite{MR277346}, since $V=\mathrm{L}[U]$,
it is the case that $k$ is an iteration of $j:V \rightarrow M$
and $M'$ is of the form $N[H]$, being a generic extension of $N$.
Recall that $V[G_\kappa]\s\bigcap_{\gamma<\delta}V[G_\kappa][g_\gamma]=V'$. So, by elementarity, we have $k[G_\kappa] \s H$.
As $\kappa$ is the critical point of $k$,
$M' \cap (V')_\kappa=(V')_\kappa=(V[G_\kappa])_\kappa$, where the second equality comes from the distributivity of each of the $g_\gamma$'s.
In addition, $(k(\mathbb{P}_\kappa)/ G_\kappa)(\kappa)=(\mathbb Q_\kappa)^{N[G_\kappa]}$ which coincides with the interpretation of $\mathbb Q_\kappa$ in $V[G_\kappa]$.
Since $\mathbb Q_\kappa$ kills the stationarity of $S_{1+\gamma}$ for at least one $\gamma$ (indeed for a tail of them),
this means that $S_{1+\gamma}$ is nonstationary in $M'$,
hence in $V^*$.
However, by Lemma~\ref{preservestationary}, $S_{1+\gamma}$ is stationary in $V[G_\kappa][g_{{\beta}}]$ for a tail of ${\beta}$'s.
In particular, $S_{1+\gamma}$ is stationary in $V'$.
If $V^*$ were a generic extension of $V'$ by some $\kappa^+$-c.c. forcing,
then $S_{1+\gamma}$ remains stationary in $V^*$. This is a contradiction.
If $V^*$ were a generic extension of $V'$ by some $\kappa$-closed forcing,
then $S_{1+\gamma}$ remains a stationary subset of $E^{\kappa^+}_\kappa$ in $V^*$. This is a contradiction.

Now we prove Clause~\eqref{precipitous}.
The argument here follows the proof of \cite[Proposition 6.13]{FMZ24}.
The partial order $\mathbb{R}^0_{\kappa}/ g_\delta$
can be viewed as $(j(\mathbb{P}_\kappa)\restriction (\kappa+1))/ (G_\kappa*g_\delta)$.
So we may assume that every condition in $\mathbb{R}^0_{\kappa}/ g_\delta$ is in $M$.
For every $p \in \mathbb{R}^0_{\kappa}/ g_\delta$,
let $f_p:\kappa \rightarrow V$ represent $p$.
Define $h:\kappa \rightarrow V[G_\kappa]$ by letting
$h(\xi)$ the be $\xi+1$ component of $G_\kappa$
for every inaccessible $\xi<\kappa$.
Then $h$ represents the $\mathbb{Q}_\kappa$-generic filter $g_0$
in the ultrapower map given by the corresponding $j^{++}$ from the proof of Claim~\ref{measurable},
i.e., $j^{++}(h)(\kappa)=g_0$.
So $p \in g_0$ iff $j^{++}(f_p)(\kappa)\in j^{++}(h)(\kappa)$ iff $\{\xi<\kappa \mid f_p(\xi)\in h(\xi)\}\in U_0$.

Now in $V'$,
define an ideal $\mathcal{I}$ over $\kappa$ in $V'$ by letting
$$\mathcal{I}:=\{ X\s\kappa\mid\mathbb{R}^0_{\kappa}/ g_\delta\text{ forces that }X \notin \dot{U}_0\}.$$
Define $e:\mathbb{R}^0_{\kappa}/ g_\delta \rightarrow \mathcal{P}(\kappa)/ \mathcal{I}$ via
$$e(p):=[\{\xi<\kappa \mid f_p(\xi)\in h(\xi)\}]_{\mathcal{I}}.$$
Then a standard duality argument shows that $e$ is a dense embedding.
Note also that $\mathbb{R}^0_{\kappa}/ g_\delta$ is $\kappa$-distributive ${<}\delta$-closed,
it follows that $\mathcal{P}(\kappa)/ \mathcal{I}$ has a $\kappa$-distributive ${<}\delta$-closed subset.
\end{proof}

\subsection{Extension of Ketonen's result}
In this short subsection, we improve Ketonen's result \cite{MR585552} that every weakly compact cardinal that carries an indecomposable ultrafilter
is Ramsey. The upcoming result is stated in terms of the Welch game $\Game^W_{\beta}$
(a variation of the games by Holy and Schlicht \cite{MR3800756} and Nielsen and Welch \cite{MR3922802}, as stated in \cite{FMZ24}).

The Welch game $\Game^W_\beta$ of length a limit ordinal $\beta$ is defined as follows:
Players I and II alternate moves:
\begin{center}
\begin{tabular}{ c||c|c|c|c|c|c }
I & $\mathcal{A}_0$& $\mathcal{A}_1$ &\dots &$\mathcal{A}_\alpha$&$\mathcal{A}_{\alpha+1}$&\dots \\
\hline
II & $U_0$ & $U_1$& \dots &$U_\alpha$&$U_{\alpha+1}$& \dots \\
\end{tabular}
\end{center}
Here, $\langle \mathcal{A}_\alpha\mid \alpha<\beta\rangle$ is a $\s$-increasing sequence of
$\kappa$-complete subalgebras of $\mathcal{P}(\kappa)$ of cardinality $\kappa$ and
$\langle U_\alpha\mid \alpha<\beta\rangle$ is a sequence of uniform $\kappa$-complete
filters, each $U_\alpha$ is a uniform ultrafilter on $\mathcal{A}_\alpha$ and $\alpha<\alpha'<\beta$ implies $U_\alpha \s U_{\alpha'}$.
We assume without loss of generality that $\mathcal{A}_0$ contains all singletons. Player~I
goes first at limit stages. The game continues until either Player~II can't
play or the play has length~$\beta$.
\medskip

\textbf{The winning condition.} Player~I wins if the game cannot continue through all stages below $\beta$.
Otherwise, Player~II wins.

\begin{lemma}\label{l59} Suppose that $\kappa$ is weakly compact and carries an indecomposable ultrafilter.
For every $\beta<\kappa^+$, Player~I has no winning strategy for $\Game^W_\beta$, the Welch game of length $\beta$.
In particular, $\kappa$ is a super Ramsey limit of super Ramsey cardinals.
\end{lemma}
\begin{proof}
Let $U$ be a uniform indecomposable ultrafilter over $\kappa$.
If $U$ is countably complete, then it is $\kappa$-complete and then Player~II will win the game $\Game^W_{\kappa^+}$
by playing the `$U$-strategy' of responding with $U\cap\mathcal A_\alpha$ to any algebra $\mathcal A_\alpha$ given by Player~I.
Thus, we may assume that $U$ is not countably complete.
By \cite[Corollary~2.6]{MR480041}, we may also assume that $U$ is weakly normal.
Now, let us recall the following fact due to Silver.
\begin{claim}[Silver, {\cite[Lemma~2]{silver}}]\label{finest}
There is a uniform ultrafilter $D$ over $\omega$ such that the ultrapower embedding $j_U: V\to M_U$ can be factored as $k\circ j_D$,
where $j_D: V\to M_D $ and $k: M_D\to M_U$ such that $k$ is $M_D$-$j_D(\kappa)$-complete, namely, for any $\sigma\in M_D$ such that $M_D\models |\sigma|< j_D(\kappa)$, we have $k(\sigma)=k``\sigma$. \qed
\end{claim}
Let $\bar{U}$ be a possibly external $M_D$-ultrafilter on $j_D(\kappa)$ derived from $k$ using $[\id]_U$.
Since $\kappa$ is weakly compact, by \cite[Lemma~3.3]{MR585552}, $\bar{U}$ is weakly amenable,
i.e., for every function $f \in M_D$ with domain $j_D(\kappa)$,
we have that $\{ \alpha \mid (\alpha<j_D(\kappa))^{M_D} \text{ and } f(\alpha) \in \bar{U} \}$ is in $M_D$.

\begin{claim} In $M_D$,
for every $\beta<j_D(\kappa^+)$, Player~I has no winning strategy for $\Game^W_{\beta}$.
\end{claim}
\begin{proof} Let $\beta<j_D(\kappa^+)$, and suppose $\sigma$ is a strategy of Player~I in $(\Game^W_{\beta})^{M_D}$.
We shall find a filter $F$ in $M_D$ such that
if Player~I follows the strategy $\sigma$, and Player~II follows the `$F$-strategy',
then Player~II wins.

The desired $F$ will be obtained by a preliminary play of the game, as follows.
Suppose $\alpha<\beta$ and $\langle (\mathcal{A}_{\xi},U_\xi)\mid \xi<\alpha\rangle$ is the outcome of a partial game in $M_D$,
where Player~I has been using $\sigma$.
In particular, for any $\xi<\alpha$, $\mathcal{A}_{\xi}$ is a $j_D(\kappa)$-algebra of size $j_D(\kappa)$.
Working in $V$, Player~II chooses the filter $U_\alpha:= \bar{U} \cap \mathcal{A}_{\alpha}$, noting that $U_\alpha \in M_D$ by the weak amenability of $\bar{U}$.

This completes the description of the preliminary game.
Now, let $F:=\bigcup_{\alpha< \beta}U_\alpha$, and note that $F\in M_D$ by the weak amenability of $\bar{U}$.
This means that in $M_D$, for the strategy $\sigma$,
Player~II can always continue the play by choosing $U_\alpha:=F\cap \mathcal{A}_{\alpha}$.
Thus $\sigma$ is not a winning strategy of Player~I in $(\Game^W_{\beta})^{M_D}$.
\end{proof}

By elementarity, in $V$, for every $\beta<\kappa^+$, Player~I has no winning strategy for $\Game^W_\beta$.
It now follows from \cite[Proposition~5.2]{MR3800756} that $\kappa$ is a super Ramsey limit of super Ramsey cardinal.
\end{proof}

\section{\texorpdfstring{$\lambda$-strongly compact cardinals}{Strongly compact cardinals}}\label{sec5}

Two concrete building blocks for the upcoming construction are the notion of forcing to add a $\kappa$-Souslin tree with $\theta$-ascent path,
and the associated forcing to add a thread through the ascent path.
It is worth mentioning that we could have used the two building blocks of the previous section or the two of the next section. The point is
that there is a variety of such pairs of blocks that fits into our framework, and we try to showcase some of them.

\begin{defn}[\cite{Git20}]\label{ascentpathdef}
Suppose $\theta<\kappa$ is a pair of infinite regular cardinals.
The forcing notion $\mathbb P(\kappa,\theta)$ consists of all pairs $\langle T,\vec{f} \rangle$ satisfying the following:
\begin{enumerate}
\item $T$ is a normal uniformly homogeneous binary tree of a successor height less than $\kappa$;
\item $\vec{f}$ is a $\theta$-ascent path through $T$.
\end{enumerate}
The ordering on $\mathbb P(\kappa,\theta)$ is defined by taking end-extensions on both coordinates.
\end{defn}

For any $\mathbb P(\kappa,\theta)$-generic filter $G$, we let $T(G):=\bigcup_{\langle T, \vec{f} \rangle \in G}T$ be the $\kappa$-tree added by $G$
and $\vec{f}^G:=\langle f^G_{\xi} \mid \xi<\kappa \rangle$ be its corresponding $\theta$-ascent path $\bigcup_{\langle T, \vec{f} \rangle \in G}\vec{f}$.
We next recall the forcing designed to thread the $\theta$-ascent path $\vec{f}^G$.
\begin{defn}[\cite{Git20}]\label{defn62}
Suppose $G$ is a $\mathbb P(\kappa,\theta)$-generic filter over $V$.
In $V[G]$, for every $\alpha<\theta$, derive the forcing $\mathbb T^{\alpha}(\kappa,\theta)$ associated with $G$:
\begin{itemize}
\item $\mathbb T^{\alpha}(\kappa,\theta)$ has underlying set $F(\kappa,\theta):=\{f^G_{\xi} \mid \xi<\kappa \}$, and
\item for all two conditions $g_0,g_1$ in $\mathbb T^{\alpha}(\kappa,\theta)$, $g_0 \leq_{\mathbb T^{\alpha}(\kappa,\theta)} g_1$ iff $g_0(\beta)\supseteq g_1(\beta)$ for every $\beta\in[\alpha,\theta)$.
\end{itemize}
Let $\mathbb{T}(\kappa,\theta)$ denote the lottery sum of $\mathbb{T}^{\alpha}(\kappa,\theta)$ over all $\alpha<\theta$.
\end{defn}

Note that if $G \s \mathbb P(\kappa,\theta)$ is generic, then in $V[G]$, $\langle \mathbb T^{\alpha}(\kappa,\theta) \mid \alpha<\theta \rangle$ is an
eventually trivial commutative projection system and $\mathbb{T}^0(\kappa,\theta)$ is ${<}\kappa$-distributive.
\begin{fact}[\cite{Git20} and {\cite[Fact~2.11]{YY23b}}]
Suppose $\theta<\kappa$ is a pair of infinite regular cardinals.
Then $\mathbb P(\kappa,\theta)*\mathbb T^{\alpha}(\kappa,\theta)$ is forcing equivalent to $\add(\kappa,1)$.\footnote{Here, $\add(\kappa,1)=\{f \subseteq \kappa \mid |f|<\kappa \}$ is the Cohen forcing to add a subset of $\kappa$.}
\end{fact}

\subsection{The first system}
Given a regular cardinal $\delta$,
an increasing sequence $\vec{\kappa}:=\langle \kappa_\beta \mid \beta<\delta \rangle$ of regular cardinals with $\delta<\kappa_0$ and a regular cardinal $\theta<\delta$,
let $\mathbb{P}(\vec{\kappa},\theta)$ be the full product forcing $\prod_{\beta<\delta}\mathbb{P}(\kappa_\beta,\theta)$
and let $G$ be $\mathbb{P}(\vec{\kappa},\theta)$-generic.
Recalling Definition~\ref{defn62}, we denote $F(\vec{\kappa},\theta):=\prod_{\beta<\delta}F(\kappa_\beta,\theta)$.
For any $\alpha<\delta$, let the forcing $\mathbb{Q}_\alpha(\vec{\kappa},\theta)$ have underlying set $F(\vec{\kappa},\theta)\restriction[\alpha,\delta)$,
and the following ordering
$$f \leq_{\mathbb{Q}_{\alpha}(\vec{\kappa},\theta)} g
\iff %
f(\beta)\leq_{\mathbb{T}^0(\kappa_\beta,\theta)}g(\beta)
\text{ for every }\beta\in[\alpha,\delta).$$
For all $\alpha\le\beta<\kappa$, let $\pi_{\alpha,\beta}$ be the projection map that sends each $f$ to $f\restriction[\beta,\delta)$.
Then $$\vec{\mathbb Q}(\vec{\kappa},\theta):=\langle \mathbb{Q}_\alpha(\vec{\kappa},\theta),\pi_{\alpha,\beta} \mid \alpha\le\beta<\delta \rangle$$
is a commutative projection system.
Let $\mathbb{Q}_\infty(\vec{\kappa},\theta)$ be its direct limit forcing.

The underlying set of $\mathbb{Q}_\infty(\vec{\kappa},\theta)$ is $F(\vec{\kappa},\theta)/ E$ for the equivalence relation $E$ over $F(\vec{\kappa},\theta)$
in which $f\mathrel{E}g$ iff there exists an $\alpha<\delta$ such that $f \restriction [\alpha,\delta)=g \restriction [\alpha,\delta)$.
The ordering is defined by leting for all $f,g \in F(\vec{\kappa},\theta)$,
$$[f]_E \leq_{\mathbb{Q}_\infty(\vec{\kappa},\theta)} [g]_E\iff\exists \alpha<\delta \forall \beta\in[\alpha,\delta)\,[f(\beta)\leq_{\mathbb{T}^0(\kappa_\beta,\theta)}g(\beta)]. $$

\begin{lemma}\label{iterationlemma} Let $\alpha<\delta$.
$\mathbb{P}(\vec{\kappa},\theta)*\dot{\mathbb{Q}}_\alpha(\vec{\kappa},\theta)$
has a dense subset that is the product of
$\mathbb{P}(\vec\kappa\restriction \alpha,\theta)$ and some ${<}\kappa_\alpha$-directed closed forcing.
\end{lemma}
\begin{proof} Let $D$ consist of all conditions $(p,\check{q})$ in
$\mathbb{P}(\vec{\kappa},\theta)*\dot{\mathbb{Q}}_\alpha(\vec{\kappa},\theta)$
such that if $p$ takes the form $\langle \langle T^\beta,\vec{f}^\beta\rangle \mid \beta<\delta \rangle$,
then $q(\beta)=\vec{f}^\beta(\h (T^\beta)-1)$ for every $\beta\in[\alpha,\delta)$.
\begin{claim} $D$ is dense in $\mathbb{P}(\vec{\kappa},\theta)*\dot{\mathbb{Q}}_\alpha(\vec{\kappa},\theta)$.
\end{claim}
\begin{proof} Let a condition $(p,\dot{q})$ in $\mathbb{P}(\vec{\kappa},\theta)*\dot{\mathbb{Q}}_\alpha(\vec{\kappa},\theta)$ be given.
Since $\mathbb{P}(\vec{\kappa},\theta)$ is $\delta$-directed closed,
we can construct a decreasing sequence $\langle p_\beta \mid \alpha \leq \beta\leq\delta\rangle$ with $p_\alpha=p$
and a sequence $\langle f_\beta \mid \alpha\leq \beta<\delta\rangle$ such that for each $\beta\in[\alpha,\theta)$,
\[
p_{\beta+1} \Vdash q(\beta)=\check{f_\beta}.
\]
We may further extend $p_\delta$ to some $p'=\langle \langle T^\beta,\vec{f}^\beta \rangle\mid \beta<\delta\rangle$ that forces
$\vec{f}^\beta(\h(T^\beta)-1)\leq_{\mathbb{T}^0(\kappa_\beta,\theta)}\check{f}_\beta$ for every $\beta\in[\alpha,\delta)$.
Set $q':=\langle \vec{f}^\beta(\h(T^\beta)-1) \mid \alpha\leq \beta<\delta \rangle$.
Then $(p',\check{q}')$ is an extension of $(p,\dot{q})$ lying in $D$.
\end{proof}

As each $(p,\check{q})\in D$ can be factored into $p\restriction \alpha$ and $(p \restriction [\alpha,\delta), q)$,
$D$ may be viewed as the product of the following two:
\begin{itemize}
\item $\mathbb{P}_{<\alpha}:=\mathbb{P}(\kappa\restriction \alpha,\theta)$, and
\item $\mathbb{R}_\alpha:=\{ (p \restriction [\alpha,\delta), q) \mid (p,\check{q})\in D\}$.
\end{itemize}

As $\mathbb{R}_\alpha$ is ${<}\kappa_\alpha$-directed closed, we are done.
\end{proof}

\subsection{The second system}
There is an another commutative projection system such that its direct limit forcing is forcing equivalent to $\mathbb{Q}_\infty(\vec{\kappa},\theta)$.
Let
$$\I(\delta,\theta):=\{ f \in {}^{\delta}\theta \mid f(\alpha)=0 \text{ for co-boundedly many } \alpha<\delta \},$$
and for all $f,g \in \I(\delta,\theta)$,
$$f \geq_{\I(\delta,\theta)} g \text{ iff }
f(\alpha)\geq g(\alpha) \text{ for every } \alpha<\delta.$$
Then $\I(\delta,\theta)$ is ${<}\theta$-directed closed.
For every $f \in \I(\delta,\theta)$, let $\mathbb{Q}^{f}(\vec{\kappa},\theta)$ be the full product $\prod_{\alpha<\delta}\mathbb{T}^{f(\alpha)}(\kappa_\alpha,\theta)$
associated with $G$.
Then
$$\langle \mathbb{Q}^{f}(\vec{\kappa},\theta) \mid f \in \I(\delta,\theta) \rangle$$
is a commutative projection system.
Let $\mathbb{Q}^{\infty}(\vec{\kappa},\theta)$ be its direct limit forcing.
Then there is an isomorphism given by the identity map $\pi: \mathbb{Q}^{\infty}(\vec{\kappa},\theta) \rightarrow \mathbb{Q}_{\infty}(\vec{\kappa},\theta)$.
For simplicity, we also view $\mathbb{Q}_{\infty}(\vec{\kappa},\theta)$ as the direct limit forcing of $\langle \mathbb{Q}^{f}(\vec{\kappa},\theta) \mid f \in \I(\delta,\theta) \rangle$.

The benefit of this commutative projection system is that $\mathbb{P}(\vec{\kappa},\theta)*\dot{\mathbb{Q}}^{f}(\vec{\kappa},\theta)$
is ${<}\kappa_0$-directed closed (see the proof of Lemma~\ref{iterationlemma}), while $\mathbb{P}(\vec{\kappa},\theta)*\dot{\mathbb{Q}}_\alpha(\vec{\kappa},\theta)$ is only ${<}\theta$-directed closed.
Thus the forcing $\mathbb{P}(\vec{\kappa},\theta)*\dot{\mathbb{Q}}^{f}(\vec{\kappa},\theta)$ can preserve more indestructible strong compactness.

\subsection{Applications}
Now we give a parallel theorem of the main theorem of \cite{YY23}. Instead of using Radin forcing, here we use the first commutative projection system mentioned above.
Curiously, by Remark~\ref{rmk37}, the Radin approach can also be presented as taking a limit ultrafilter in an intersection of a sequence of models (the ones associated with the generic Radin sequence), see \cite[\S5]{MR2630840} or \cite[\S4]{hayut2023prikry}.

\begin{thm}\label{single}
Suppose that a regular cardinal $\delta$ carries a $\lambda$-complete uniform ultrafilter,
and there exists an increasing sequence $\langle \kappa_\alpha \mid \alpha<\delta \rangle$
of supercompact cardinals with $\kappa_0>\delta$.
Then there exists a model in which $\sup_{\alpha<\delta}\kappa_\alpha$ is the least
$\lambda$-strongly compact cardinal.
\end{thm}
\begin{proof} By performing a preparatory forcing \cite{MR0472529} if necessary,
we may assume that for every $\alpha<\delta$, $\kappa_\alpha$ is indestructible under any ${<}\kappa_\alpha$-directed closed forcing.
Write $\vec{\kappa}$ for $\langle \kappa_\alpha \mid \alpha<\delta \rangle$,
and $\kappa$ for $\sup_{\alpha<\delta}\kappa_\alpha$.
Let $G \s \mathbb{P}(\vec{\kappa},\omega)$ be generic over $V$.
Then by our analysis above, we have the commutative projection system
$\vec{\mathbb Q}(\vec{\kappa},\omega)=\langle \mathbb{Q}_\alpha(\vec{\kappa},\omega),\pi_{\alpha,\beta} \mid \alpha\le\beta<\delta \rangle$ derived from $G$ in $V[G]$.
Let $H_0 \s \mathbb{Q}_0(\vec{\kappa},\omega)$ be generic over $V[G]$.
Let $H_\alpha \s \mathbb{Q}_\alpha(\vec{\kappa},\omega)$ be the filter generated by $\pi_{0,\alpha}[H_0]$ for every $\alpha<\delta$,
and let $H_\infty$ be be the filter generated by $\pi_{0,\infty}[H_0]$.
Then $H_\alpha$ is generic over $V[G]$ for each $\alpha<\delta$, and so is $H_\infty$.

Now we prove that $V[G][H_\infty]$ is our desired model in which
$\kappa$ is the least $\lambda$-strongly compact cardinal.
Since $\mathbb{Q}_0(\vec{\kappa},\omega)$ is $\delta$-distributive, by Fact~\ref{thmintersection},
we have
$$V[G][H_\infty]=\bigcap_{\alpha<\delta}V[G][H_\alpha].$$

\begin{claim} Let $\alpha<\delta$. Then $\kappa_\alpha$ is supercompact in $V[G][H_\alpha]$.
\end{claim}
\begin{proof}
By Lemma~\ref{iterationlemma},
we know that $\mathbb{P}(\vec{\kappa},\omega)*\dot{\mathbb{Q}}_\alpha(\vec{\kappa},\omega)$
has a dense subset that is the product of
$\prod_{\beta<\alpha}\mathbb{P}(\kappa_\beta,\omega)$ and a ${<}\kappa_\alpha$-directed closed forcing, say $\mathbb R_\alpha$.
Since
$\kappa_\alpha$ is indestructible under any ${<}\kappa_\alpha$-directed forcing
and $\mathbb R_\alpha$ is ${<}\kappa_\alpha$-directed closed,
it follows that $\kappa_\alpha$ is supercompact in $V[G]^{\mathbb R_\alpha}$.
Note also that $\prod_{\beta<\alpha}\mathbb{P}(\kappa_\beta,\omega)$ is a small forcing
of size less than $\kappa_\alpha$,
it follows that $\kappa_\alpha$ is supercompact in $V[G][H_\alpha]$.
\end{proof}
Thus by Clause~\eqref{maincor3} of Corollary~\ref{maincor1},
$\kappa$ is a $\lambda$-strongly compact cardinal in the intersection model $\bigcap_{\alpha<\delta}V[G][H_\alpha]=V[G][H_\infty]$.

Now we only need to prove that there is no $\lambda$-strongly compact cardinal below $\kappa$.
Since $\mathbb{P}(\kappa_\alpha,\omega)$ adds a $\kappa_\alpha$-Aronszajn tree with an $\omega$-ascent path
and $\mathbb{P}(\vec{\kappa},\omega)$ forces that $\mathbb{\dot{Q}}_\infty(\vec{\kappa},\omega)$ is ${<}\kappa$-distributive,
we know that there is a $\kappa_\alpha$-Aronszajn tree with an $\omega$-ascent path for every $\alpha<\delta$ in $V[G][H_\infty]$.
Thus by Fact~\ref{killstrongcompactness},
there is no $\lambda$-strongly compact cardinal below $\kappa$.
\end{proof}
\begin{remark} Let us sketch an alternative proof of the above theorem via a lifting argument.
Write $H$ for $H_\infty$.
We only prove that $\kappa$ is $\lambda$-strongly compact in $V[G][H]$.
Let $W$ be a weakly normal $\lambda$-complete uniform ultrafilter over $\delta$,
and let $i_W:V \rightarrow M_W$ be the corresponding ultrapower map given by $W$.
Since $\mathbb{P}(\vec{\kappa},\omega)*\dot{\mathbb{Q}}_\infty(\vec{\kappa},\omega)$ is $\delta$-distributive,
we may lift $i_W$ and obtain an elementary embedding $i_W^+:V[G,H] \rightarrow M_W[G^*,H^*]$ by a transfer argument.
Here $G^**H^* \subseteq i_W(\mathbb{P}(\vec{\kappa},\omega)*\dot{\mathbb{Q}}_\infty(\vec{\kappa},\omega))$ is the filter generated by $i_W[G*H]$.

Denote $i_W(\langle \kappa_\alpha \mid \alpha<\delta \rangle)$ by
$\langle \bar{\kappa}_\alpha \mid \alpha<i_W(\delta) \rangle$.
Let $\bar{\delta}:=\sup(i_W[\delta])$.
Since $W$ is weakly normal, $\bar{\delta}=[\id]_W$.

Let $\bar{\mathbb{Q}}:=i_W(\langle \mathbb Q_\alpha(\vec{\kappa},\omega) \mid \alpha<\delta \rangle)(\bar{\delta})$
and let $\bar{\pi}_{\bar{\delta},\infty}:\bar{\mathbb{Q}} \rightarrow i_W(\mathbb{\dot{Q}}_\infty)$ be the projection
induced by $i_W(\langle \pi_{\alpha,\beta}\mid \alpha\leq \beta<\delta\rangle)$.
Then in $M_W[G^*]$,
$\bar{\mathbb{Q}}$ is the full product $(\mathbb{T}^0(\bar{\kappa}_\alpha,\theta))^{M_W[G^*]}$ over all $\bar{\delta}\leq\alpha<i_W(\delta)$.
By elementarity of $i_W$,
we have that the supercompactness of $\bar{\kappa}_\alpha$ is indestructible under any ${<}\bar{\kappa}_\alpha$-directed forcing for every $\alpha<i_W(\delta)$.
Define
$$F^*:=\{ f \in \bar{\mathbb{Q}} \mid \exists h \in i_W^+[p] \text{ for some } p\in H, [h \restriction [\bar{\delta},i_W(\delta)) \leq_{\bar{\mathbb Q}} f] \}.$$
Clearly, $F^*\subseteq \bar{\mathbb{Q}}$ is a filter and $\bar{\pi}_{\bar{\delta},\infty}[F^*] \s H^*$.
Therefore, $F^* \s \bar{\mathbb{Q}}/ H^*$ is an $M_W[G^*,H^*]$-generic set.
To see this, take any dense open subset $D$ of $\bar{\mathbb{Q}}$ in $M_W[G^*,H^*]$.
Let $g$ represent $D$.
We may assume that
$g(\alpha)$ is a dense open subset of $\mathbb Q_\alpha(\vec{\kappa},\omega)$ for every $\alpha<\delta$.
Then $\pi_{0,\alpha}^{-1}[g(\alpha)]$ is a dense open subset of $\mathbb Q_0(\vec{\kappa},\omega)$ for every $\alpha<\delta$.
Thus $\bigcap_{\alpha<\delta} \pi_{0,\alpha}^{-1}[g(\alpha)]$ is a dense open subset of $\mathbb Q_0(\vec{\kappa},\omega)$
by $\delta$-distributivity of $\mathbb Q_0(\vec{\kappa},\omega)$.
Pick an $f \in \bigcap_{\alpha<\delta} \pi_{0,\alpha}^{-1}[g(\alpha)]\cap \pi_{0,\infty}^{-1}[H]$.
Then $i_W(f)\restriction [\delta,i_W(\delta)) \in D \cap F^*$.
So $F^* \s \bar{\mathbb{Q}}/ H^*$ is $M_W[G^*,H^*]$-generic.

Let $\eta > \kappa$ be an arbitrary regular cardinal in $V[G][H]$.
Then it is regular in $M_W$.
Since the supercompactness of $\bar{\kappa}_{\bar{\delta}}$ in $M_W$ is indestructible,
we have that $\bar{\kappa}_{\bar{\delta}}$ is supercompact in $M_W[G^*,H^*,F^*]$.
Then we may have a $\eta$-supercompact embedding $j':M_W[G^*,H^*,F^*] \rightarrow N'$ for some transitive model $N'$.
Let $j: M_W[G^*,H^*] \rightarrow N$ for some transitive model $N$ be $j'\restriction M_W[G^*,H^*]$.
Then $N'$ is a $j(\bar{\mathbb Q})$-generic extension of $N$.
Since $j(\bar{\mathbb Q})$ is $\eta$-distributive,
so $N'$ thinks that $N$ is closed under $\eta$-sequence.
Note also that $j[\eta] \in N'$ since $j'$ is an $\eta$-supercompact embedding, we have $j[\eta]\in N$.

We claim that $j \circ i_W^{+}$ witnesses that $\kappa$ is $\lambda$-strongly compact up to $\eta$.
To see this, let $X:=j[\eta] \supseteq j[i_W^{+}[\eta]]$. Then $X \in N$.
Note also that $N \vDash |X|=\eta<j(i_W^{+}(\kappa))$,
we have $j \circ i_W^{+}$ witnesses that $\kappa$ is $\lambda$-strongly compact up to $\eta$.
So $\kappa$ is $\lambda$-strongly compact cardinal in $V[G][H]$ since $\eta$ was chosen arbitrarily. $\hfill\blacksquare$
\end{remark}

Next we prove that there may exist a proper class of singular $\delta$-strongly compact cardinals for different $\delta$'s simultaneously.
\begin{thm}\label{classmany}
Suppose there are proper class many supercompact cardinals.
Then there exists a model in which
there exist proper class many measurable $\delta$
such that the least $\delta$-strongly compact cardinal is singular.
\end{thm}
\begin{proof}
Suppose $\vec{\kappa}:=\langle\kappa_\alpha \mid \alpha \in \ord \rangle$
is an increasing and discrete sequence of supercompact cardinals.
By performing a preparatory forcing if necessary,
we may assume that the supercompactness of $\kappa_\alpha$ is indestructible
under ${<}\kappa_\alpha$-directed closed forcing for every $\alpha$.

By recursion,
we may take a proper class $\Delta$ consisting of measurable cardinals such that
$\delta$ is the least measurable cardinal above $\sup_{\gamma \in \Delta \cap \delta}\kappa_{\gamma}$
for every $\delta \in \Delta$.
For every $\delta \in \Delta$,
let $\theta_\delta:=(\sup_{\gamma \in \Delta \cap \delta}\kappa_{\gamma})^+$.
Let
$$\mathbb{P}_\delta:=(\prod_{\theta_\delta<\alpha<\delta}\mathbb{P}(\kappa_\alpha,\theta_\delta))*\mathbb{\dot{Q}}_{\infty}(\vec{\kappa}\restriction \delta,\theta_\delta).$$
Let $\mathbb{P}$ be the Easton support product $\prod_{\delta \in \Delta}\mathbb{P}_\delta$.
Though $\mathbb{P}$ is a class forcing,
the standard Easton argument shows $V^{\mathbb{P}}\vDash \zfc$.

We claim that $V^{\mathbb{P}}$ is our desired model in which
$\sup_{\alpha<\delta}\kappa_\alpha$ is the least $\delta$-strongly compact cardinal for every $\delta \in \Delta$.
To see this,
take any $\delta \in \Delta$. Let $\kappa:=\sup_{\alpha<\delta}\kappa_\alpha$.
Note that the Easton support product $\mathbb{P}_{>\delta}:=\prod_{\gamma \in \Delta\setminus (\delta+1)}\mathbb{P}_{\gamma}$
is ${<}\kappa$-directed closed,
it follows that in $V^{\mathbb{P}_{>\delta}}$,
for each $\alpha$ with $\theta_\delta \leq \alpha<\delta$,
$\kappa_\alpha$ is supercompact and indestructible under any ${<}\kappa_\alpha$-direct closed forcing.
Thus $\kappa$ is the least $\delta$-strongly compact cardinal
in $V^{\mathbb{P}_{>\delta}\times \mathbb{P}_\delta}$ by the argument in Theorem \ref{single}.
Note also that $\mathbb{P}_{<\delta}:=\prod_{\gamma \in \Delta \cap \delta}\mathbb{P}_{\gamma}$ is a small forcing of size less than $\delta$,
it follows that $\kappa$ is the least $\delta$-strongly compact cardinal in $V^{\mathbb{P}}=V^{\mathbb{P}_{>\delta}\times \mathbb{P}_\delta \times \mathbb{P}_{<\delta}}$.
\end{proof}

Next we deal with another case
that all $\delta$'s are below $\kappa_0$.
We will use a similar forcing as before.
The difference is that the argument relies also on the second commutative projection system with index set $\I(\delta,\theta)$.

\begin{thm}\label{thm6.9}
Suppose $\Delta$ is a discrete set of measurable cardinals
such that there is a $\delta$-complete uniform ultrafilter over $\prod_{\delta' \in \Delta \setminus (\delta+1)}\I(\delta',(\sup(\Delta \cap \delta'))^+)$ for every $\delta \in \Delta$,
and $\vec{\kappa}:=\langle \kappa_\alpha \mid \alpha<\sup(\Delta) \rangle$ is an increasing sequence of supercompact cardinals with $\kappa_0>\sup(\Delta)$.
Then there exists a model in which the singular cardinal $\sup_{\alpha<\delta}\kappa_{\delta}$ is the least $\delta$-strongly compact cardinal for every $\delta \in \Delta$.
\end{thm}
\begin{proof}
By performing a preparatory forcing if necessary,
we may assume that the supercompactness of $\kappa_\alpha$ is indestructible
under any ${<}\kappa_\alpha$-directed closed forcing for every $\alpha<\sup(\Delta)$.
Let $\theta_\delta:=(\sup(\Delta \cap \delta))^+$ for each $\delta \in \Delta$.
For each $\delta \in \Delta$, let
$$\mathbb P_\delta:=(\prod_{\theta_\delta \leq \alpha<\delta}\mathbb{P}(\kappa_\alpha,\theta_\delta))*\mathbb{\dot{Q}}_{\infty}(\vec{\kappa}\restriction \delta,\theta_\delta).$$
Let $\mathbb{P}$ be the full product $\prod_{\delta \in \Delta}\mathbb P_\delta$.
Then for every $\delta \in \Delta$,
$$\mathbb{\vec{Q}}(\delta):=\langle (\prod_{\theta_\delta \leq \alpha<\delta}\mathbb{P}(\kappa_\alpha,\theta_\delta))*\mathbb{\dot{Q}}^{f}(\vec{\kappa}\restriction \delta,\theta_\delta) \mid f \in \I(\delta,\theta_\delta) \rangle$$
is a commutative projection system
such that each forcing $\mathbb{P}(\kappa_\alpha,\theta_\delta))*\mathbb{\dot{Q}}^{f}(\vec{\kappa}\restriction \delta,\theta_\delta)$
in it is ${<}\kappa_{\theta_\delta}$-directed closed
and $\mathbb P_\delta$ is its direct limit forcing.

Now we show $V^{\mathbb{P}}$ is our desired model in which the singular cardinal $\sup_{\alpha<\delta}\kappa_{\delta}$ is the least $\delta$-strongly compact cardinal for every $\delta \in \Delta$.
To see this, take any $\delta \in \Delta$. Let $\kappa=\sup_{\alpha<\delta}\kappa_\alpha$.
Let $\mathbb{P}^{{<}\delta}:=\prod_{\delta' \in \Delta \cap \delta}\mathbb P_{\delta'}$,
let $\mathbb{P}^{{\geq}\delta}:=\prod_{\delta' \in \Delta \setminus \delta}\mathbb P_{\delta'}$ and let
$$\I:=\delta \times \prod_{\delta' \in \Delta \setminus (\delta+1)} \I(\delta',\theta_{\delta'}).$$
Then $|\I|<\kappa_0$ by our assumptions that $\kappa_0>\sup(\Delta)$ is inaccessible.
Let
$$\mathbb{\vec{R}}(\vec{\kappa}\restriction \delta,\theta_\delta):=\langle (\prod_{\theta_\delta \leq \alpha<\delta}\mathbb{P}(\kappa_\alpha,\theta_\delta))*\mathbb{\dot{Q}}_{\alpha}(\vec{\kappa}\restriction \delta,\theta_\delta)\mid \alpha<\delta \rangle.$$
Then its direct limit forcing is $\mathbb P_\delta$.

Now consider the product commutative projection system
$$\mathbb{\vec{R}}(\vec{\kappa}\restriction \delta,\theta_\delta)\times \prod_{\delta' \in \Delta \setminus (\delta+1)}\mathbb{\vec{Q}}(\delta').$$
Its direct limit forcing is $\mathbb{P}^{{\geq}\delta}$.
Take any forcing, say $\mathbb{P}'^\delta\times \mathbb{P}'^{{>}\delta}$, in the product commutative projection system.
Here $\mathbb{P}'^{\delta}$ is the forcing $(\prod_{\theta_\delta \leq \alpha<\delta}\mathbb{P}(\kappa_\alpha,\theta_\delta))*\mathbb{\dot{Q}}_{\beta}(\vec{\kappa}\restriction \delta,\theta_\delta)$
for some $\beta$ with $\theta_\delta\leq \beta<\delta$
and $\mathbb{P}'^{{>}\delta}$ is a forcing in $\prod_{\delta' \in \Delta \setminus (\delta+1)}\mathbb{\vec{Q}}(\delta')$.
Note that forcing the $\mathbb{P}^{{>}\delta}$ is ${<}(2^{\kappa_\delta})^+$-directed closed,
a similar argument as in Theorem \ref{classmany} shows that
$\kappa_{\alpha}$ is supercompact in $V^{\mathbb{P}^{{<}\delta}\times \mathbb{P}'^\delta\times \mathbb{P}'^{{>}\delta}}$.

Since there is a $\delta$-complete uniform ultrafilter over $\I$ by our assumption,
by the $\kappa_0$-distributivity of $\mathbb{P}'^\delta\times \mathbb{P}'^{{>}\delta}$,
it follows by Theorem~\ref{menas-sheard} that $\kappa$ is $\delta$-strongly compact in the direct limit forcing extension $V^{\mathbb{P}}$.
Note also that there is a $\kappa_\alpha$-Souslin tree with a $\theta_\delta$-ascent path for every $\alpha \in [\theta_\delta,\delta)$,
it follows that $\kappa$ is the least $\delta$-strongly compact in $V^{\mathbb{P}}$.
\end{proof}

\section{The $C$-sequence number}\label{sec6}
In \cite{paper35}, Lambie-Hanson and Rinot
introduced a new cardinal characteristic $\chi(\kappa)$ to measure how far a cardinal $\kappa$ is from being weakly compact, as follows.
\begin{defn}[The $C$-sequence number of $\kappa$] \label{c-seq_num_def}
If $\kappa$ is weakly compact, then let $\chi(\kappa):=0$. Otherwise, let
$\chi(\kappa)$ denote the least (finite or infinite) cardinal $\chi\le\kappa$
such that, for every $C$-sequence $\langle C_\beta\mid\beta<\kappa\rangle$,\footnote{Recall that $\langle C_\beta\mid\beta<\kappa\rangle$
is a \emph{$C$-sequence} iff for every $\beta<\kappa$, $C_\beta$ is a closed subset of $\beta$ satisfying $\sup(C_\beta)=\sup(\beta)$.}
there exist $\Delta\in[\kappa]^\kappa$ and $b:\kappa\rightarrow[\kappa]^{\chi}$
with $\Delta\cap\alpha\s\bigcup_{\beta\in b(\alpha)}C_\beta$
for every $\alpha<\kappa$.
\end{defn}

By \cite[Corollary~2.6]{paper35}, if $\kappa$ is an inaccessible cardinal and $\square(\kappa)$ holds, then $\chi(\kappa)=\kappa$.
The main result of this section is Theorem~\ref{thm77} below,
asserting that for every weakly compact cardinal $\kappa$,
for every infinite cardinal $\delta<\kappa$,
there is a cofinality-preserving forcing extension in which $\chi(\kappa)=\delta$.

\subsection{Two building blocks}
Let us recall some content from Lambie-Hanson's \cite{HH17}, namely the indexed variation of $\square(\kappa,\delta)$ and the notion of forcing to add it by initial segments.
\begin{defn}
Let $\delta < \kappa$ be infinite regular cardinals. $\mathcal{C} = \langle C_{\alpha, i} \mid \alpha < \kappa, i(\alpha) \leq i < \delta \rangle$ is a $\square^{\mathrm{ind}}(\kappa, \delta)$-sequence iff all of the following hold:
\begin{enumerate}
\item For all $\alpha < \kappa$, $i(\alpha) < \delta$.
\item For all limit $\alpha < \kappa$ and $i(\alpha) \leq i < \delta$, $C_{\alpha, i}$ is club in $\alpha$.
\item For all limit $\alpha < \kappa$ and $i(\alpha) \leq i < j < \delta$, $C_{\alpha, i} \subseteq C_{\alpha, j}$.
\item For all limit $\alpha < \beta < \kappa$ and $i(\beta) \leq i < \delta$, if $\alpha \in \acc(C_{\beta, i})$, then $i(\alpha) \leq i$ and $C_{\beta, i} \cap \alpha = C_{\alpha, i}$.
\item For all limit $\alpha < \beta < \kappa$, there is $i(\beta) \leq i < \delta$ such that $\alpha \in \acc(C_{\beta, i})$.
\item There is no club $D \subseteq \kappa$ such that, for all $\alpha \in \acc(D)$, there is $i(\alpha) \leq i < \delta$ such that $D \cap \alpha = C_{\alpha, i}$.
\end{enumerate}

The principle $\square^{\mathrm{ind}}(\kappa, \delta)$ asserts that there is a $\square^{\mathrm{ind}}(\kappa, \delta)$-sequence.
\end{defn}

\begin{defn}\label{indexedsquaredef}
Suppose $\delta < \kappa$ are infinite regular cardinals. Let $\mathbb{S}(\kappa, \delta)$ be a forcing poset whose conditions are all $p = \langle C^p_{\alpha, i} \mid \alpha \leq \gamma^p, i(\alpha)^p \leq i < \delta \rangle$ satisfying the following conditions.
\begin{enumerate}
\item $\gamma^p < \kappa$ is a limit ordinal and, for all $\alpha \leq \gamma^p, i(\alpha)^p < \delta$.
\item For all limit $\alpha \leq \gamma^p$ and all $i(\alpha)^p \leq i < \delta$, $C^p_{\alpha, i}$ is a club in $\alpha$.
\item For all limit $\alpha \leq \gamma^p$ and all $i(\alpha)^p \leq i < j < \delta$, $C^p_{\alpha, i} \subseteq C^p_{\alpha, j}$.
\item For all limit $\alpha < \beta \leq \gamma^p$ and all $i(\beta)^p \leq i < \delta$, if $\alpha \in \acc(C^p_{\beta, i})$, then $i(\alpha)^p \leq i$ and $C^p_{\beta, i} \cap \alpha = C^p_{\alpha, i}$.
\item For all limit $\alpha < \beta \leq \gamma^p$, there is $i(\beta)^p \leq i < \delta$ such that $\alpha \in \acc(C^p_{\beta, i})$.
\end{enumerate}

For $p,q \in \mathbb{S}(\kappa, \delta)$, let $q \leq p$ iff $q$ end-extends $p$.
\end{defn}
Suppose $G \s \mathbb S(\kappa,\delta)$ is $V$-generic.
Then $\bigcup G$ is an $\square^{\mathrm{ind}}(\kappa, \delta)$-sequence,
which we shall denote by
$\mathcal{C}(G):=\langle C_{\alpha,i}\mid \alpha<\kappa,i(\alpha)\leq i<\delta \rangle$.

We now recall the forcing that threads $\mathcal{C}(G)$.

\begin{defn}
Let $\delta < \kappa$ a pair of infinite regular cardinals, and let $i < \delta$.
$\mathbb{T}_i(\mathcal{C}(G))$ is the forcing poset whose conditions are all $C_{\alpha, i}$ such that $\alpha < \kappa$ is a limit ordinal and $i(\alpha) \leq i$. $\mathbb{T}_i(\mathcal{C}(G))$ is ordered by end-extension.

Let $\mathbb{T}(\mathcal{C}(G))$ be the lottery sum of all $\mathbb{T}_i(\mathcal{C}(G))$ over all $i<\delta$.
\end{defn}
Clearly, $\langle \mathbb{T}_i(\mathcal{C}(G)),\pi_{i,j} \mid i\leq j<\delta \rangle$ is a commutative projection system
that is $\delta$-distributive and eventually trivial.
Here $\pi_{i,j}:\mathbb{T}_i(\mathcal{C}(G)) \rightarrow \mathbb{T}_j(\mathcal{C}(G))$
is the canonical projection for all $i\leq j<\delta$.

\begin{fact}[{\cite[Theorem~3.4]{paper35}}]\label{fact75}
Suppose that $\kappa$ is a weakly compact cardinal and $\delta$ is an infinite regular cardinal smaller than $\kappa$.
Then there is a cofinality-preserving forcing extension in which $\chi(\kappa)=\delta$.
\end{fact}

We now give an abstract argument for the inequality $\chi(\kappa)\leq\delta$, without assuming that $\delta$ is regular.
\begin{lemma}\label{maincor}
Suppose:
\begin{itemize}
\item $\kappa$ is a regular cardinal,
\item $\I$ is a directed poset of some infinite cofinality $\delta<\kappa$,
\item $\vec{\mathbb P}:=\langle \mathbb{P}_i, \pi_{i,j} \mid i\leq_{\I} j \rangle$ is
an eventually trivial commutative projection system and $\mathbb{P}_0$ is $\delta$-distributive,
\item for every $i\in\I$, $\mathbb{P}_i$ forces that $\chi(\kappa)\le\delta$.
\end{itemize}
Then $\chi(\kappa)\leq \delta$.
\end{lemma}
\begin{proof} Without loss of generality, we may assume that $|\I|=\delta$.
By \cite[Lemma~2.2(2)]{paper35}, we may also assume that $\kappa$ is not the successor of a regular cardinal, so that $\delta^+<\kappa$.
Take any $C$-sequence $\vec{C}=\langle C_\beta\mid \beta<\kappa\rangle$.
Assuming that for every $i\in\I$, $\mathbb{P}_i$ forces that $\chi(\vec C)\le\delta$,
so by \cite[Lemma~2.1]{paper35}, in $V^{\mathbb P_i}$, there exists a club $D_i\s\kappa$ such that, for every $\alpha<\kappa$,
$D_i\cap (E^\alpha_{>\delta})^V$ is covered by the union of no more than $\delta$ many elements of $\vec{C}$.
Note that we may assume that $D_i \s D_j$ whenever $i<_{\I}j$, because we could replace $D_i$ by $\bigcap_{k\ge_{\I}i}D_k$.
Thus $D:=\bigcup_{i \in \I}D_i$ is in $\bigcap_{i \in \I} V[G_i]=V$
and it is the case that $D\cap (E^\kappa_{>\delta})^V$ is a cofinal subset of $\kappa$ all of whose proper initial segments are covered by the union of no more than $\delta$ many elements of $\vec{C}$.
\end{proof}

\subsection{Application}
Our goal here is to extend Fact~\ref{fact75}, proving that $\chi(\kappa)$ can be an arbitrary infinite cardinal $\delta<\kappa$.
To this end, instead of adding a $\square^{\ind}(\kappa,\delta)$-sequence (as done in the proof of \cite[Theorem~3.4]{paper35}),
we shall be adding $\delta$-many $\square^{\ind}(\kappa,\omega)$-sequences,
and show that $\chi(\kappa)$ may be guaranteed to be $\delta$. This yields Theorem~\ref{thmc}, as follows.

\begin{thm}\label{thm77}
Suppose $\kappa$ is a weakly compact cardinal,
and $\delta<\kappa$ is an infinite cardinal.
Then there exists a forcing extension in which $\chi(\kappa)=\delta$.
\end{thm}
\begin{proof}
We may assume that the weak compactness of $\kappa$ is indestructible under $\add(\kappa,1)$ (cf.~\cite{MR0495118}).
Consider the directed poset $\I$ with underlying set $$\{ f \in {}^\delta\omega \mid f \text{ has only finitely many nonzero points}\},$$
ordered by the everywhere-dominance ordering. Clearly, $\cf(\I)=\delta$.

Let $\mathbb{S}\s \prod_{\xi<\delta}\mathbb{S}(\kappa,\omega)$ be the poset consisting of all
$p \in \prod_{\xi<\delta}\mathbb{S}(\kappa,\omega)$ with
$\gamma^p:=\gamma^{p(\xi)}$ for every $\xi<\delta$, and for every $\alpha\leq\gamma^p$,
$\langle i(\alpha)^{p(\xi)} \mid \xi<\delta \rangle\in \I$.

Let $G:=\prod_{\xi<\delta} G_\xi \s \mathbb{S}$
be $V$-generic.
For each $\xi<\delta$, let
$$\vec{C}^\xi:=\langle C^\xi_{\alpha,i}\mid i^\xi(\alpha)\leq i<\omega, \alpha<\kappa\rangle$$
be the $\square^{\ind}(\kappa,\delta)$-sequence given by $G_\xi$.
We claim that in $V[G]$, $\chi(\kappa)=\delta$.

Consider $\vec{C}:=\langle C_\alpha \mid \alpha<\kappa \rangle$, where $C_\alpha:=\bigcap_{\xi<\delta}C^\xi_{\alpha,i^\xi(\alpha)}$.
\begin{claim}
$\chi(\vec{C})\geq \delta$.
\end{claim}
\begin{proof}
Suppose not.
Then there exists a $\Delta \in [\kappa]^\kappa$ and
a $b:\kappa \rightarrow [\kappa]^{\theta}$ for some $\theta<\delta$
such that for every $\alpha<\kappa$, $\Delta \cap \alpha \s \bigcup_{\beta\in b(\alpha)}C_\beta$.
Note that for every $\alpha<\kappa$,
$|b(\alpha)|=\theta$ and for each $C_\beta$, there are only finitely many $\xi<\delta$
such that $i^\xi(\beta)\neq 0$,
it follows that
the set $B_\alpha:=\{ \xi<\delta \mid \exists \beta \in b(\alpha)(i^\xi(\beta)\neq 0) \}$ has size $\leq \theta$.
Note that $\kappa>\theta$ is inaccessible,
we may assume that there is a $B \s \delta$ of size $\leq \theta$
such that $B:=B_\alpha$ for every $\alpha<\kappa$.
Since $\theta<\delta$,
it follows that there exists an $\xi<\delta$ such that $\xi \notin B$.
This means that for every $\alpha<\kappa$,
$i^{\xi}(\beta)=0$ for every $\beta \in b(\alpha)$.

Now work in $V$. Let $\dot{\Delta}$ and $\dot{b}$ be the name of $\Delta$ and $b$, respectively.
Let $p$ force that $\dot{\Delta}$, $\dot{b}$ and $\check{\xi}$ with $\xi<\delta$ witness together the above fact.
We can recursively construct a decreasing sequence $\langle p_\alpha \mid \alpha<\theta^+\rangle$ satisfying the following requirements:
\begin{itemize}
\item $p_0=p$;
\item For all $\alpha<\theta^+$, there is a $\beta_\alpha$ with $\gamma^{p_\alpha}<\beta_{\alpha}<\gamma^{p_{\alpha+1}}$
such that $p_{\alpha+1}$ forces that $\beta_{\alpha} \in \Delta$;
\item For all $\alpha<\theta^+$, $i^{\xi}(\gamma^{p_\alpha})^{p_\alpha(\xi)}=1$,
and $i^{\eta}(\gamma^{p_\alpha})^{p_\alpha(\eta)}=0$ for every $\eta\neq \xi$.
\end{itemize}
Now define a condition $q$ by letting $\gamma^q:=\sup\{\gamma^{p_\alpha}\mid \alpha<\theta^+\}$,
$i^{\xi}(\gamma^{q})^{q(\xi)}=1$,
and $i^{\eta}(\gamma^{q})^{q(\eta)}=0$ for every $\eta \neq \xi$,
and for all $\eta<\delta$ and all $i^\eta(\gamma^q)\leq i<\omega$, $C^\eta_{\gamma^{q(\eta)},i}:=\bigcup_{\alpha<\theta^+}C^{p_\alpha(\eta)}_{\gamma^{p_\alpha(\eta)},i}$.
Then $q$ forces that $\Delta \cap \gamma^q$ is unbounded in $\gamma^q$.
However, for every $\beta \in b(\gamma^q)$, we have $i^{\xi}(\beta)=0$ since $q$ extends $p$,
it follows that
$C_{\beta}=\bigcap_{\eta<\delta}C^\eta_{\beta,i^\eta(\beta)} \cap \gamma^q \s C^\xi_{\beta,0}\cap \gamma^q$.
Since $i^\xi(\beta)=0<1=i^\xi(\gamma^q)$,
we have $\gamma^q \notin C^\xi_{\beta,0}$ by the definition of indexed square.
Thus $C_{\beta}\s C^\xi_{\beta,0}\cap \gamma^q$ is a bounded subset of $\gamma^q$.
So $\Delta \cap \gamma^q \s \bigcup_{\beta \in b(\gamma^q)}C_{\beta} \cap \gamma^q$ is a bounded subset of $\gamma^q$ since $\cf(\gamma^q)=\theta^+>\theta$,
a contradiction.
\end{proof}
For the other inequality, work in $V[G]$.
Let $\mathbb{T}_f:=\prod_{\eta<\delta}\mathbb{T}_{f(\eta)}(\vec{C}^\eta)$ for every $f \in \I$.
Then $\langle \mathbb{T}_f \mid f \in \I\rangle$ is an eventually trivial commutative projection system,
and so by Corollary~\ref{maincor}, $\chi(\kappa)\leq \cf(\I)=\delta$.
Thus $\chi(\kappa)=\delta$.
\end{proof}

We remark that the conclusion of Lemma~\ref{maincor} is optimal
in the sense that an actual equality need not hold. Indeed, a construction demonstrating this will appear in \cite{paper70}.

\section*{Acknowledgments}
The first and second authors were supported by the European Research Council (grant agreement ERC-2018-StG 802756)
and by the Israel Science Foundation (grant agreement 203/22).
The third author was supported by a UKRI Future Leaders Fellowship [MR/T021705/2].

Some of the results of this paper were presented by the third author
during the Set Theory Workshop at the Academy of Mathematics and Systems Science, Chinese Academy of Sciences,
on Set Theory Seminar in Beijing Normal University,
on Nankai Logic Colloquium, April 2024.
The results of this paper were presented by the second author
on Set Theory Seminars of BIU and of HUJI in May 2024,
and in a poster session at the \emph{120 Years of Choice} conference in Leeds, July 2024.
We thank the corresponding organizers for the opportunity to present this
work and the participants for stimulating discussions.

We are grateful to Tom Benhamou for detecting and informing us of an error in the previous version of Lemma~\ref{l59}.

\end{document}